\documentclass[reqno,12pt,letterpaper]{amsart}
\usepackage{amsmath,amssymb,amsthm,graphicx,mathrsfs,url}
\usepackage[usenames,dvipsnames]{color}
\usepackage[colorlinks=true,linkcolor=Red,citecolor=Green]{hyperref}
\usepackage{amsxtra}
\usepackage{tikz}
\usepackage{amscd}
\usetikzlibrary{arrows.meta}
\usepackage[normalem]{ulem}
\usepackage{wasysym} 
\usepackage{graphicx}
\usepackage{subcaption}

\usepackage{graphicx,color}

\def\arXiv#1{\href{http://arxiv.org/abs/#1}{arXiv:#1}}
\usepackage{array}
\newcolumntype{P}[1]{>{\centering\arraybackslash}m{#1}}

\def\wrtext#1{\relax\ifmmode{\leavevmode\hbox{#1}}\else{#1}\fi}

\def\?[#1]{\textbf{[#1]}\marginpar{\Large{\textbf{??}}}}

\def\smallsection#1{\smallskip\noindent\textbf{#1}.}
\let\epsilon=\varepsilon 

\newcommand{\RR}{{\mathbb R}}

\newcommand{\CC}{{\mathbb C}}
\newcommand{\TT}{{\mathbb T}}
\newcommand{\ZZ}{{\mathbb Z}}

\newtheorem*{thm}{Theorem}
\newtheorem{prop}{Proposition}

\newtheorem{lemm}[prop]{Lemma}

\numberwithin{equation}{section}

\DeclareMathOperator{\Spec}{Spec}

\let\Re=\Real

\DeclareMathOperator{\supp}{supp}

\usepackage{scalerel}

\newcommand\reallywidehat[1]{\arraycolsep=0pt\relax%
\begin{array}{c}
\stretchto{
  \scaleto{
    \scalerel*[\widthof{\ensuremath{#1}}]{\kern-.5pt\bigwedge\kern-.5pt}
    {\rule[-\textheight/2]{1ex}{\textheight}} 
  }{\textheight} %
}{0.5ex}\\           
#1\\                 
\rule{-1ex}{0ex}
\end{array}
}

\begin{document}

\title[Discrete vs.~continuous in the semiclassical limit]{Discrete vs. continuous in the semiclassical limit}

\author{Simon Becker}
\address[Simon Becker]{ETH Zurich, 
Institute for Mathematical Research, 
Rämistrasse 101, 8092 Zurich, 
Switzerland}
\email{simon.becker@math.ethz.ch}

\author{Jens Wittsten}
\address[Jens Wittsten]{Department of Engineering, University of Bor{\aa}s, SE-501 90 Bor{\aa}s, Sweden}
\email{jens.wittsten@hb.se}

\author{Maciej Zworski}
\address[Maciej Zworski]{Department of Mathematics, University of California,
Berkeley, CA 94720, USA}
\email{zworski@math.berkeley.edu}

\begin{abstract}
We compare the bottom of the spectrum of discrete and continuous Schrödinger operators with periodic potentials with barriers at the boundaries of their fundamental domains (see Figure~\ref{fig:barrier}). Our results show that these energy levels coincide in the semiclassical limit and we provide an explicit rate of convergence. We demonstrate the optimality of our results by using Bohr-Sommerfeld quantization conditions for potentials exhibiting non-degenerate wells, and by numerical experiments for more general potentials. 
We also investigate the dependence of the spectrum of the discrete semiclassical Schr\"odinger operator on the semiclassical parameter $h$ and show that it can be discontinuous. 
\end{abstract}

\maketitle

\section{Introduction and statement of results}

Suppose that $ V \in C^\infty ( \mathbb R^n; \mathbb R ) $ is periodic with respect to $ \mathbb Z^n $:
$ V ( x + \gamma ) = V ( x ) $, $ \gamma \in \mathbb Z^n $. We make the following general 
assumptions on $ V $ (see Figure~\ref{fig:barrier}):
\begin{equation}
\label{eq:defV}
\text{$\exists $ a fundamental domain of $ \mathbb Z^n $,  $ F$, 
such that $ V|_{\partial F }  >  \min V $.} 
\end{equation}
To $ V$ we associate discrete and continuous Schr\"odinger operators,
\begin{equation}
\label{eq:Pd}
(P_{\rm{d}} ( h ) v )( \gamma ) :=  \sum_{ j=1}^n ( 2v ( \gamma ) - v ( \gamma + e_j ) - v ( \gamma - e_j )  ) +
V ( h \gamma ) v ( \gamma ), \ \ \gamma \in \mathbb Z^n , 
\end{equation}
and
\begin{equation}
\label{eq:Pc}
P_{\rm{c}} ( h ) u ( x )  :=  - \Delta u (x ) + V ( h x ) u ( x) , \ \ x \in \mathbb R^n ,
\end{equation}
which are  selfadjoint on $ \ell^2 ( \mathbb Z^n )$ and $ L^2 ( \mathbb R^n )$ (with domain $ H^2 ( \mathbb R^n ) $), respectively.

Motivated by questions considered by Detherage--Stier--Srivastava \cite{DSS} (we refer to that 
paper for background and pointers to related work) we provide the following 
general result:

\begin{thm} 
\label{thm:main} Suppose that $ V $ satisfies \eqref{eq:defV} and $ P_{\bullet } $ are defined
in \eqref{eq:Pd} and  \eqref{eq:Pc}. Then 
\begin{equation}
\label{eq:comp}    d(h) := \frac{ \min \Spec ( P_{\rm{d}} ( h ) ) }{ \min \Spec ( P_{\rm{c}} ( h ) ) } = 1 + \mathcal O ( h ) . 
\end{equation}
\end{thm}

\begin{figure}
{\begin{tikzpicture}
\node at (-6,0) {\includegraphics[width=16cm]{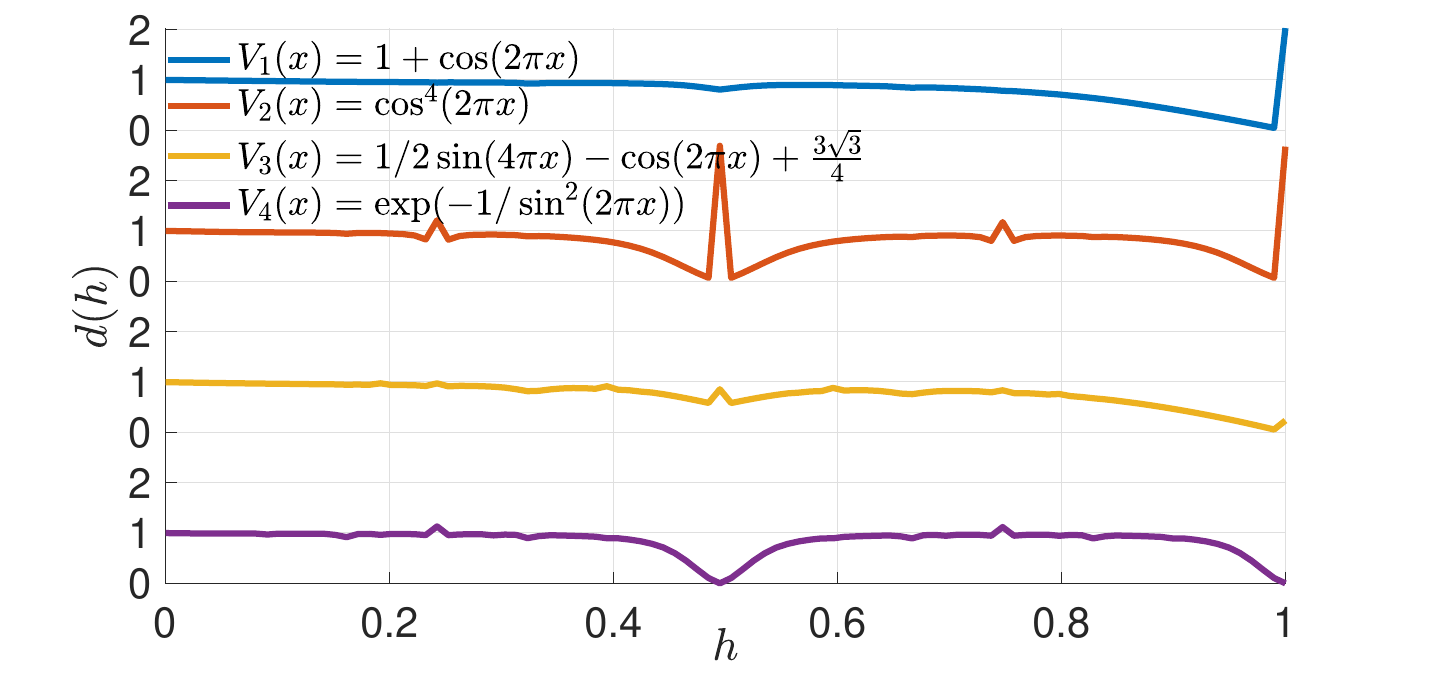}};
\end{tikzpicture}}
\caption{\label{fig:rational}Plots of $ d ( h ) $ in \eqref{eq:comp}
sampled at rational values, $ h \in \mathbb Q$ for four different potentials. 
The results of \cite{DSS} show that  
$ d ( h ) \leq 1 $ for $h \notin \mathbb Q $. The spikes indicate dramatic discontinuities of the spectrum of $P_{\rm{d}} ( h )$ at
rational points -- see \S \ref{s:con} -- where $ d ( h ) > 1$.}
\end{figure}

As a consequence, the bottom of the spectrum of the discrete operator $P_{\rm{d}}(h)$ can be determined by studying the (often more tractable) continuous Schrödinger operator $P_{\rm{c}}(h)$, and vice versa.
The proof of the theorem follows standard semiclassical arguments, similar to, though much simpler than,  those
in Helffer--Sj\"ostrand \cite{HelSj} and Becker--Zworski \cite[\S 5]{BZ}.

\begin{figure}
\begin{centering}
{\begin{tikzpicture}
\node at (-6,0) {\includegraphics[width=16cm]{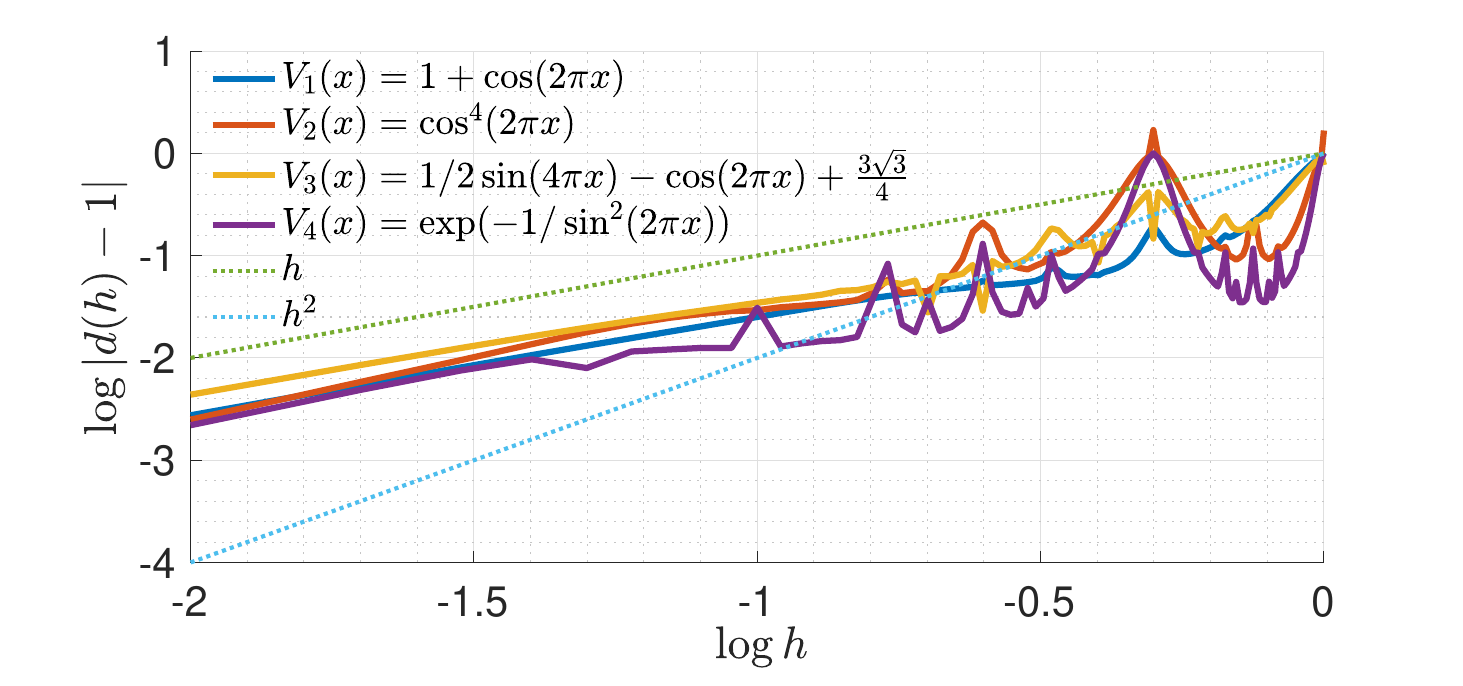}};
\end{tikzpicture}}
\end{centering}
\caption{\label{f:optimal} The $\log\log $ plot corresponding to Figure~\ref{fig:rational} indicating that $ \mathcal O (h ) $ in Theorem \ref{thm:main} 
is optimal.}
\end{figure}

When more structure of $ V $ is known, for instance when $ V $ has a unique non-degenerate minimum in the fundamental domain then one can replace the bound in \eqref{eq:comp} by a full asymptotic expansion -- see \S \ref{s:BS} for a detailed analysis in dimension one. That shows that the error bound in \eqref{eq:comp} is optimal -- see  Figure \ref{f:optimal}. We also remark that the spectrum of $ P_{\rm{c}} $ is absolutely continuous while the
spectrum of $ P_{\rm{d}} $ may have a very complicated structure depending on the rationality properties of $h$. In particular, we can see discontinuities in $ d ( h ) $ -- see \S \ref{s:dis}.

\section{Localization of spectra of periodic operators}\label{sec:2}

Instead of considering the operators given in \eqref{eq:Pd} and \eqref{eq:Pc}, we will consider 
operators unitarily equivalent to $ P_{\rm{d}} $ and $ P_{ \rm{c}} $ respectively: with 
$ \mathbb T^n :=  \mathbb R^n/ ( 2 \pi \mathbb Z)^n$, 
\begin{equation}
\label{eq:defHP} 
\begin{gathered} H := \sum_{ j=1}^n 2(1 - \cos x_j)  + V ( h D_x ) , \ \ x \in \mathbb T^n ,  \ \ \ H : L^2 ( \mathbb T^n ) \to L^2 ( \mathbb T^n ) , 
 \\
P := - h^2 \Delta + V ( x ) , \ \ x \in \mathbb R^n , \ \ \ 
P : H^2 ( \mathbb R^n ) \to L^2 ( \mathbb R^n ) .
 \end{gathered}
\end{equation}
Without loss of generality we can assume that  $ \min V  = 0 $ and that $ \bar F = 
\{ x:  0 \leq x_j  \leq 1 \} $.

Assumption \eqref{eq:defV} implies that we can choose an open subset, $ \Omega \Subset \bar F  $, such that $ V |_{ \bar 
F \setminus \Omega } > c_0 > 0 $. We can also choose an open neighbourhood (in $\mathbb R^n$), 
$ \Gamma $ of $ \partial F $ such that $ V|_{ \Gamma } > c_0 > 0 $. The image of $\Gamma$ under $V$ (or just $\Gamma$) will be called the barrier. This is illustrated in Figure \ref{fig:barrier}.

We will use the following set of localizing functions:
\begin{equation}
\label{eq:defchi} 
\begin{gathered} \chi , \tilde \chi,  \psi  \in C_{\rm{c}}^\infty ( \mathbb R^n; [0,1] ) , \ \ 
\sum_{ \gamma \in \mathbb Z^n } \chi ( \xi - \gamma ) = 1, \ \ \chi|_\Omega = 1, \\ 
\supp \nabla \chi, \, \supp \nabla \tilde \chi, \, \supp \nabla \psi \subset \Gamma, \ \  \widetilde \chi |_{\supp \chi } = 1 , \ \ \ 
\psi|_{ \supp \widetilde \chi } = 1. 
\end{gathered}
\end{equation}
We then define 
\begin{equation}
\label{eq:V0}
V_0 ( x ) := V ( x ) + 1 - \psi ( x ) ,  \ \ \ V_0 ( x ) > \min( c_0, 1 )  ,  \ \  x \notin \Omega 
\end{equation} 
and the corresponding operators
\begin{equation}
\label{eq:defHP0} 
\begin{gathered} H_0 := \sum_{ j=1}^n 2(1 - \cos x_j)  + V_0 ( h D_x ) , \ \ x \in \mathbb T^n ,  \ \ \ H_0 : L^2 ( \mathbb T^n ) \to L^2 ( \mathbb T^n ) , 
 \\
P_0 := - h^2 \Delta + V_0 ( x ) , \ \ x \in \mathbb R^n , \ \ \ 
P_0 : H^2 ( \mathbb R^n ) \to L^2 ( \mathbb R^n ) .
 \end{gathered}
\end{equation}
These operators have discrete spectrum in a neighbourhood of $ 0 $ (see the proof of 
Lemma \ref{l:ground} for that standard fact). 

The key result now is given in 
\begin{lemm}
\label{l:local}
In the notation of \eqref{eq:defHP} and \eqref{eq:defHP0} 
\begin{equation}
\label{eq:H2H0} \begin{split} 
& \Spec( H) \cap [0, c_0 /2 ] \subset \Spec (H_0 ) + (- \varepsilon( h ), \varepsilon ( h )) , \\ 
&  \Spec( P) \cap [0, c_0 /2 ] \subset \Spec (P_0 ) + (- \varepsilon( h ), \varepsilon ( h )) ,  
 \end{split} \end{equation}
 where $\varepsilon(h) = \mathcal O ( h^\infty ) $ .
 \end{lemm} 
\begin{proof} We will prove the first assertion in \eqref{eq:H2H0}. The second one can be proved similarly.
It is more standard and also follows from Floquet theory and the existence of a barrier. We will use
semiclassical pseudodifferential calculus on $ \mathbb T^n $ -- see \cite[\S 5.3.1]{z12} for a presentation which allows us to cite results from \cite[Chapter 4]{z12}.

\begin{figure}
\begin{centering}
\includegraphics[width=0.95\textwidth,trim={0.7cm 0.4cm 0.35cm 0.5cm},clip]{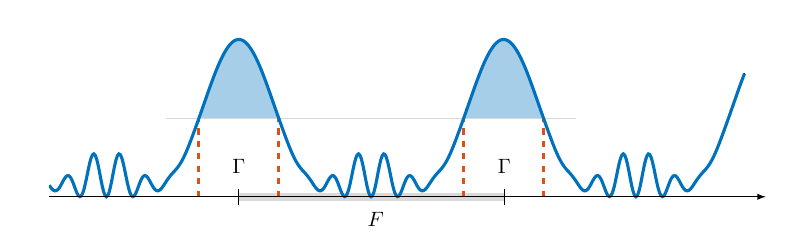}
\end{centering}
\caption{Graph of a potential $V$ with a barrier, over a fundamental domain $F\subset \RR$ (indicated in grey), showing a neighbourhood $\Gamma$ of the boundary $\partial F$ such that $V|_{\Gamma}>\min V$. The barrier $V(\Gamma)$ is shaded blue. \label{fig:barrier}}
\end{figure}

Let $z\in\CC $ and set
$$
Q:=H -z,\quad Q_0:= H_0 -z.
$$
Let $r_\gamma:L^2(\TT)\to L^2(\TT)$ be the multiplication operator $r_\gamma u(x)=e^{i \gamma x/h}u(x)$
and $ \bullet_\gamma ( \xi ) = \bullet ( \xi - \gamma ) $, $ \bullet = \chi, \tilde \chi, \psi $ (from \eqref{eq:defchi}). 
Then 
 \begin{equation}
 \label{eq:rgamma}  r_\gamma Q = Q r_\gamma ,    \ \ \  r_\gamma \bullet ( hD ) = \bullet_\gamma( h D ) r _\gamma, \ \ 
 \bullet= \chi, \tilde \chi, \psi. \end{equation}
As a candidate for an approximate inverse of $ Q $ we introduce
$$
F:=\sum_{\gamma \in\ZZ^n}\tilde\chi_\gamma (hD)(Q_0^\gamma)^{-1}\chi_\gamma (hD),\quad Q_0^\gamma:=r_\gamma Q_0r_{-\gamma}.
$$
Since $\lVert (Q_0^\gamma)^{-1}\rVert_{L^2\to L^2}=d(z,\Spec(H_0))^{-1}$ we have 
\begin{equation*}
F=\mathcal O(d(z,\Spec(H_0))^{-1}).
\end{equation*}
From \eqref{eq:rgamma}, we see that 
$$
Q\tilde\chi_\gamma (hD)=r_\gamma(Q_0-(1-\psi (hD)))r_{-\gamma}\tilde\chi_\gamma (hD)=(Q_0^\gamma-(1-\psi_\gamma (hD))\tilde\chi_\gamma (hD).
$$
The assumptions that $ \psi =1 $ on the support of $ \tilde \chi $, and
$ \tilde \chi = 1 $ on the support of $ \chi $, and $ \sum_\gamma \chi_\gamma = 1 $ (see \eqref{eq:defchi}) give
\begin{equation*}
\begin{aligned}
QF&=I +\sum_{\gamma\in\ZZ^n}[Q_0^\gamma,\tilde\chi_\gamma (hD)](Q_0^\gamma)^{-1}\chi_\gamma (hD).
\end{aligned}
\end{equation*}
We now modify $Q_0$ to make it invertible without changing it on the support of $\nabla\tilde\chi$. To this end, let $ \varphi\in C_{\rm{c}}^\infty(\RR^n; [0,1])$ satisfy (in the notation of \eqref{eq:defchi})
\[   \supp \varphi \subset \Gamma, \ \ \ \varphi|_{ \supp \nabla \tilde \chi } = 1. \]
We then put 
$$
H_1 = H_0 + 1-\varphi (hD),\quad Q_1=H_1 -z.
$$
Since all the terms in the symbol of $H_1$ (also denoted $H_1$ for simplicity), are non-negative, and $ 2 - \psi (\xi ) - \varphi ( \xi ) \geq 2 $ for $ \xi \notin \Gamma$
and $ V |_\Gamma > c_0 $, the symbol of $ H_1 $ satisfies
$$
H_1(x,\xi)=2 \sum_{j=1}^n (1-\cos(x_j))+V(\xi)+1-\psi(\xi)+1- \varphi(\xi)\ge c>0.
$$
Hence, \cite[Theorem 4.29]{z12} gives the existence of $ Q_1^{-1} $ (with a bound independent of 
$h $) for $ z < c_1 $, $ c_1> 0 $ and $ 0 < h < h_0 $. 
Writing $Q_1^\gamma=r_\gamma Q_1r_{-\gamma}$ we have
$$
(Q_0^\gamma)^{-1}=(Q_1^\gamma)^{-1}-(Q_1^\gamma)^{-1}(Q_0^\gamma-Q_1^\gamma)(Q_0^\gamma)^{-1}
$$
which gives
\begin{equation*}
\begin{aligned}
QF&=I+\sum_{\gamma\in\ZZ^n}[Q_0^\gamma,\tilde\chi_\gamma (hD)](Q_1^\gamma)^{-1}\chi_\gamma (hD)
\\&\quad
-\sum_{\gamma\in\ZZ^n}[Q_0^\gamma,\tilde\chi_\gamma (hD)](Q_1^\gamma)^{-1}(Q_0^\gamma-Q_1^\gamma)(Q_0^\gamma)^{-1}\chi_\gamma (hD).
\end{aligned}
\end{equation*}
Since the symbol of $[Q_0^\gamma,\tilde\chi_\gamma (hD)]$ has support contained in $\supp\nabla\tilde\chi_\gamma$ and $Q_0^\gamma=Q_1^\gamma$ there, the second sum is $\mathcal O(h^\infty d(z,\Spec(H_0 ))^{-1})$ since $\sum_{\gamma} \chi_\gamma =1$. Hence,
$$
QF=I+\sum_{\gamma\in\ZZ^n}A_\gamma+\mathcal O(h^\infty d(z,\Spec(H_0))^{-1})_{L^2 \to L^2} ,
$$
where
\begin{gather*}
A_\gamma=[Q_0^\gamma,\tilde\chi_\gamma (hD)](Q_1^\gamma)^{-1}\chi_\gamma (hD),\quad Q_j^\gamma=r_\gamma Q_jr_{-\gamma}.
\end{gather*}
To bound the sum of $ A_\gamma $ we will use the Cotlar--Stein Lemma, see \cite[Theorem C.5]{z12}
and for that we need to estimate the norms of $ A_\gamma A_\rho^* $ and $ A_\gamma^* A_\rho $.

By construction, $\supp\chi_\gamma\cap\supp\chi_\rho=\emptyset$ unless $\rho=\gamma+\sum_{j=1}^n a_je_j$ with $a_j\in\{-1,0,1\}$. In particular, the supports are disjoint if $|\gamma-\rho|>\sqrt{n}$
and hence  $\chi_\gamma ( h D )  \chi_\rho ( h D) =0$ if $|\gamma-\rho|>\sqrt{n}$. Since 
$$
A_\gamma A_\rho^*=[Q_0^\gamma,\tilde\chi_\gamma (hD)](Q_1^\gamma)^{-1}\chi_\gamma (hD)\chi_\rho (hD)(Q_1^\rho)^{-1}[Q_0^\rho,\tilde\chi_\rho (hD)]^*
$$
(if $ z \in \mathbb R $, $ Q_1^\rho $ is selfadjoint and to keep the notation simple we make that 
assumption)
this means that $A_\gamma A_\rho^*=0$ if $|\gamma-\rho|>{\sqrt{n}}$. For $|\gamma-\rho|\le {\sqrt{n}}$ we note that modulo terms of 
size $ O (h^\infty ) $, the symbol of $[Q_0^\gamma,\tilde\chi_\gamma (hD)]$ has support contained in $\supp\nabla\tilde\chi_\gamma$ where $\chi_\gamma\equiv0$ -- 
see \cite[Theorem 4.25]{z12}. 
Since $Q_1$ is invertible we find for fixed $\gamma $ that
$$
\sum_{\rho \in\ZZ^n}\lVert A_\gamma A_\rho^*\rVert^{1/2}=\sum_{ {|\rho-\gamma|\le\sqrt{n}}} \lVert A_\gamma A_\rho^*\rVert^{1/2}=\mathcal O(h^\infty),
$$
and the bound is uniform in $ \gamma $.
This gives $\sup_\gamma\sum_{\rho \in\ZZ^n}\lVert A_\gamma A_\rho^*\rVert^{1/2}=\mathcal O(h^\infty)$.

Next, we consider
$$
A_\gamma^*A_\rho=\chi_\gamma (hD) (Q_1^\gamma)^{-1} [Q_0^\gamma,\tilde\chi_\gamma (hD)]^*[Q_0^\rho,\tilde\chi_\rho (hD)](Q_1^\rho)^{-1}\chi_\rho (hD).
$$
We will show that 
\begin{equation}
\label{eq:4CS}
\lVert A_\gamma^*A_\rho\rVert^{1/2}=\mathcal O(\langle \gamma -\rho \rangle^{-\infty}h^\infty),
\end{equation}
which implies  that $\sup_\gamma\sum_{\rho \in\ZZ^n}\lVert A_\gamma^*A_\rho\rVert^{1/2} = \mathcal O(h^{\infty})$. 

To see \eqref{eq:4CS}, we first consider the case of $ | \gamma - \rho | \leq  {\sqrt{n}} $. Then, 
as in the analysis of $ A_\gamma A_\rho^* $, 
$ [Q_0^\rho,\tilde\chi_\rho (hD)](Q_1^\rho)^{-1}\chi_\rho (hD) = \mathcal O ( h^\infty )_{ L^2 \to L^2 } $, 
uniformly in $\rho $, giving \eqref{eq:4CS}. When $ | \gamma - \rho | >  {\sqrt{n}} $,  we note that 
$ \chi_\gamma = \mathcal O ( m_\gamma^{-N} ) $ for any $N$, where 
$ m_\gamma $ is an order function (in the sense of \cite[\S 4.4]{z12}) given by 
$ m_\gamma ( x, \xi ) := ( 1 + | \xi - \gamma |^2 )^{\frac12} $. Composition formula for 
pseudodifferential operators \cite[Theorem 4.18]{z12}
then give $ A_\gamma A_\rho^* = B_{\rho \gamma} ( x, hD , h) $, 
$ B_{\rho \gamma } \in S ( \mathcal O ( h^\infty )  m_\gamma^{-N} m_\rho^{-N} ) $.
Since $ \sup_{\xi \in \mathbb R^n } m_\gamma^{-N} m_\rho^{-N} \leq  {2^{N/2}} \langle \rho - \gamma \rangle^{-N} $ (Peetre's inequality), \eqref{eq:4CS} follows.

Hence the assumptions of the the Cotlar--Stein Lemma \cite[Theorem C.5]{z12} are satisfied 
and we conclude that 
$$
QF=1+R,\quad \| R\|_{L^2 \to L^2} \leq  \tfrac12 \varepsilon ( h ) d(z,\Spec(H_0 ))^{-1}\ \ \ 
\varepsilon ( h ) = \mathcal O ( h^\infty ) .
$$
This shows that 
$$
d ( z, \Spec ( H_0 )) \geq \varepsilon (h ) \ \Longrightarrow \ 
(H-z)^{-1}=Q^{-1} = F(1+R)^{-1} , 
$$
that is $ z \notin \Spec ( H) $, proving the first claim in \eqref{eq:H2H0}.
\end{proof}

\section{Proof of the comparison result}

We will now study the ground states of $ H_0 $ and $ P_0 $. They are localized to $ x = 0 $, $\xi\in\Omega$, and $ \xi = 0 $, $x\in \Omega $, respectively:

\begin{lemm}
\label{l:loc}
Suppose that $ \chi \in S(\mathbb R^{2n},  1 )  $ (see \cite[\S 4.4]{z12}),  $\chi \geq  0 $,
$ \supp \chi \subset \mathbb R^n \times B_{\mathbb R^n }  ( 0, 1 ) $, 
  and for some 
 $ c_j > 0 $, $ j=0,1 $, 
\[  V_0 ( x ) <  c_0 \ \text{and} \    |\xi| <  c_1 \ \Longrightarrow \ \chi ( x, \xi ) = 1 . \]
If  
\[  (  H_0 - \lambda ) u_1 = u_0 , \ \ \ ( P_0 - \lambda ) w_1 = w_0 , \ \ \lambda = \mathcal O ( h ) , \]
then, with $ R ( x, \xi ) = ( \xi, x ) $, 
\begin{equation}
\label{eq:loc} 
\begin{split}
u_1 & = ( R^*\chi)^{\rm{w}} ( x, h D ) u_1 + \mathcal O ( \| u_1 \|  h^\infty )_{L^2}  + \mathcal O ( \| u_0 \| )_{L^2 } , \ \ L^2 = L^2 ( \mathbb T^n ) ,\\
   w_1 & = \chi^{\rm{w}} ( x, h D ) w_1 + \mathcal O ( \| w_1 \|  h^\infty )_\mathscr S  + \mathcal O ( \| w_0 \| )_{L^2 } , \ \ L^2 = L^2 ( \mathbb R^n ) ,  
\end{split}
\end{equation}
where we identified $ B_{\mathbb R^n } ( 0 , 1 ) $ with a subset of $ \mathbb T^n := \mathbb R^n / ( 
2 \pi \mathbb Z )^n $. 
\end{lemm}

\begin{proof}
Let us consider the case of $ H_0 $: since $ 2 \sum_{j=1}^n ( 1 - \cos x_j ) + V_0 ( \xi ) + R^* \chi (x, \xi ) 
> c_2 > 0  $, we see that $ H_0 + ( R^* \chi)^{\rm{w}}  ( x, h D) - \lambda $ is invertible for $ h $ small enough,
with bounds independent of $ h$ (see \cite[Theorem 4.29]{z12}). Writing
 $ \chi^{\rm{w}} := \chi^{\rm{w}} ( x, h D ) $ we then have $ u_1 = u_2 + u_3$, where
\begin{equation}
\label{eq:u02u1}  \begin{split}  
&  u_2 = ( H_0 + ( R^* \chi)^{\rm{w}}  - \lambda)^{-1}  ( R^* \chi)^{\rm{w}} u_1 , \ \ u_3 = ( H_0 + ( R^* \chi)^{\rm{w}}  - \lambda)^{-1} u_0   
 . \end{split} \end{equation}
 If $ \widetilde \chi $ has the same properties as $ \chi $ and $ \widetilde \chi = 1$
 on $ \supp \chi $, then the composition formula \cite[Theorem 4.18]{z12} shows that 
  \[  ( 1 - ( R^* \widetilde \chi)^{\rm{w}})( H_0 + ( R^* \chi)^{\rm{w}}   - \lambda)^{-1} ( R^* \chi)^{\rm{w}}  = \mathcal O (h^\infty)_{ L^2 \to L^2 } , \]
 that is, in the notation of \eqref{eq:u02u1}, $ ( R^* \widetilde \chi)^{\rm{w}}  u_2 = u_2 + \mathcal O (h^\infty  \| u_1\|)_{L^2} $. 
 Hence,
 \[ \begin{split} ( R^* \widetilde \chi)^{\rm{w}} u_1& = ( R^* \widetilde \chi)^{\rm{w}} u_2 + 
 ( R^* \widetilde \chi)^{\rm{w}} u_3  = 
 u_2 + \mathcal O ( h^\infty \| u_1 \|)_{L^2} + \mathcal O ( \| u_0 \| )_{L^2} \\
 & = 
 u_1 + \mathcal O ( h^\infty \| u_1 \|)_{L^2} + \mathcal O ( \| u_0 \| )_{L^2} . \end{split} \]
This gives the first statement in \eqref{eq:loc} with $ \chi $ replaced by  $ \widetilde \chi $. 
The argument for $ P_0 $ is the same.
\end{proof}

To compare the 
spectra it is convenient to make another modification and consider an operator on $ \mathbb R^n$ 
whose ground state is within $ \mathcal O ( h^\infty ) $ of $ H_0 $. We define it as follows:
let 
\begin{equation}\label{eq:A}
\begin{gathered}  A ( \xi ) := 2 \sum_{j=1}^n ( 2 - \cos \xi_j  - \psi_0 ( \xi_j ) ) , \\ \psi_0 \in 
C_{\rm{c}}^\infty ( (-1,1)  ; [ 0 , 1 ] ) , \ \
\psi_0|_{ (-\frac12, \frac12) } = 1, \ \ \psi_0 ( - t ) = \psi_0 ( t ) . \end{gathered}
\end{equation}
We then put
\begin{equation}
\label{eq:P1}   P_1 := A ( h D_x ) + V_0 ( x ) . \end{equation}
If $ w $ is the ground state of $ H_0 $ then Lemma \ref{l:loc} shows that it is localized and hence, 
after taking the semiclassical Fourier transform, produces
a quasimode for $ P_1$ (i.e., a solution to $ ( P_1 - \lambda ) u = \mathcal O ( h^\infty)_{L^2}  $,  
$ \|u \| = 1 $).  Similarly, a ground state of $ P_1$ produces
a quasimode for $ H_0 $
(using an appropriate periodization argument, cf.~the proof of \cite[Corollary 1.4]{BW}).
Hence, 
\[ \min \Spec ( H_0 ) = \min \Spec ( P_1 ) + \mathcal O ( h^\infty ) . \]

To obtain a comparison with $ \min \Spec (P_0)  $ we record the following lemma:
\begin{lemm}
  \label{l:ground}
  Suppose that $ V_0 \in C^\infty ( \mathbb R^n ) $, $ \min V_0 = 0 $,
  $ V_0 ( x ) > c_0 > 0 $ for $  |x| \geq R $.  Then, there exists
  $ C_0 $ and $ h_0 $ such that for $ 0 < h < h_0 $, 
  \begin{equation*}
    \Spec ( - h^2 \Delta + V_0  )  \cap [ 0 , C_0 h ]  =
  \Spec_{\rm{pp}} ( - h^2 \Delta + V_0  )  \cap [ 0 , C_0 h ] . \end{equation*}
  If $ \lambda (h) := \min \Spec_{\rm{pp}} ( - h^2 \Delta + V_0 ) $ and
  $ ( - h^2 \Delta + V_0  ) u = \lambda ( h ) u $, $ \| u \| = 1 $, then
  \begin{equation}
    \label{eq:der}
 h^2/C_0 \leq \lambda ( h ) \leq C_0 h  \ \ \text{ and }    \ \ 
  \| ( h D)^2 u \|^2 \leq C h \lambda ( h ) .
  \end{equation}
  \end{lemm}
\begin{proof}
The fact that the spectrum near zero is discrete is standard:
 suppose that $ V_1 ( x ) = V_0 ( x ) $
for $ |x| \geq R $ and $ V_1 ( x )> c_0 $ everywhere. Then $ R ( \lambda ) := 
( - h^2 \Delta + V_1 ( x ) - \lambda )^{-1} : L^2 ( \mathbb R^n ) \to H^2_h ( \mathbb R^n ) $ exists
for $ \Re \lambda < c_0 $ if $ h $ is small enough: see \cite[Theorems 4.29, 7.1]{z12}. Then
\begin{equation}
\label{eq:resol}  - h^2 \Delta + V_0  - \lambda  = ( - h^2 \Delta + V_1  - \lambda ) ( I + R ( \lambda ) ( V_0 - V_1 ) ) .\end{equation}
Since for $ \Re \lambda < c_0 $, $ R ( \lambda ) ( V_0 - V_1 )$ is a compact operator on $ L^2 $ (for instance we can consider it as a pseudodifferential operator using \cite[\S 8.1]{z12} and then use
\cite[Theorem 4.28]{z12}) it follows from  \cite[Theorem D.4]{z12} that 
$ \lambda \mapsto ( I + R ( \lambda ) ( V_0 - V_1 ) )^{-1} $ is meromorphic in $ \Re \lambda < c_0 $.
That means that the resolvent of $  - h^2 \Delta + V_0$ is meromorphic there and the spectrum is discrete.

That $ h^2/C_0  < \lambda ( h ) \leq C_0 h $ is  equally standard.
To see the upper bound, suppose that $ \psi \in C^\infty_{\rm{c}} ( B ( 0 ,1 ) ) $ and $ \int_{\mathbb R^n } | \psi ( x ) |^2 dx = 1. $
Suppose that $ V_0 ( x_0) = 0 = \min V_0 $ and assume without loss of generality that $ x_0 = 0$.  Then $ V_0 ( x ) = \mathcal O ( |x|^2 ) $. 
Define $ \psi_h ( x) := h^{-\frac n4} \psi ( h^{-\frac12} x ) $. Then 
\[  \begin{split}   ( - h^2 \Delta + V_0 ( x ) ) \psi_h ( x )   & = h \left( - h^{-\frac n4} \Delta \psi (h^{-\frac12} x ) + 
h^{-\frac n4} \mathcal O ( ( |h^{-\frac12} x |^2  ) \psi ( h^{-\frac 12 }  x  ) \right)
=: h \widetilde \psi_h .\end{split}  \]
Since $ \|\psi_h \| = 1 $ and $ \|\widetilde \psi_h \| \leq C_0 $ for some constant $ C_0 $, this shows
$ \min_{ \| u \| = 1 } \langle ( - h^2 \Delta + V_0 ) u , u \rangle \leq C_0 h $, that is $ \lambda(h) \leq C_0 h$. 

For the lower bound on $ \lambda ( h ) $, we note that the localization Lemma \ref{l:local} applies to 
$ - h^2 \Delta + V_0 ( x ) $ and hence we can replace $ u $ by an approximate mode supported in 
$ |x| < R_0 $ and 
satisfying $ ( - h^2 \Delta + V_0 ( x ) - \lambda (h )) u = \mathcal O ( h^\infty )_{L^2}  $. 
In particular, $   \| x_j u \|_{L^2} \leq   R_0    $. 
Hence, from the uncertainty principle (see \cite[Theorem 3.9]{z12}), 
\[ \| h D_{x_j}  u \|_{L^2} \geq \frac h { 2\| x_j u \|} \geq h/ C_1. \]
Together with $ V_0 \geq 0 $, the equation for (the new) $ u $ gives
\begin{equation}
\label{eq:energy}  \| h D u \|^2 + \| V_0^{\frac12} u \|^2_{L^2} = \lambda( h ) + \mathcal O ( h^\infty) , \ \ \ \ 
D :=\tfrac1 i ( \partial_{x_1} , \ldots , \partial_{x_n } ),
\end{equation}
and this shows that $ \lambda( h) \geq h^2/ C_0 $ (after replacing $C_0$ with $\max(C_0,C_1^2)$ if necessary). 

To obtain the second part of \eqref{eq:der}, we consider the equation satisfied by $ h D_{x_k } u $:
\[   - h^2 \Delta ( h D_{x_k } u ) + V_0 h D_{x_k } u + h (D_{x_k } V_0) u = \lambda (h ) hD_{x_k } u . \]
Pairing the two sides with $ h D_{x_k} u$, integrating by parts in the first term, and then using
$ | D_{x_k } V_0 | \leq C V_0^{\frac12} $ and \eqref{eq:energy}  gives 
\begin{equation}
\label{eq:pair} \begin{split} 
 \| h D_{x_k} h D u \|^2  & \leq \lambda ( h ) \| h D_{x_k} u \|^2 - h \langle  (D_{x_k } V_0) u, h D_{x_k}u \rangle \\
 & \leq  \lambda(h)^2 + C h \| V^{\frac12}_0 u \| \| h D_{x_k } u \| \leq \lambda( h )^2 + 
C  h \lambda ( h )\\
&  \leq  ( C + C_0 ) h \lambda( h ) ,
\end{split}  \end{equation}
which gives \eqref{eq:der}.
\end{proof}

We also need an analogue of this lemma for $ P_1 $ defined in \eqref{eq:P1}:
\begin{lemm}
  \label{l:groundP1}
  Let $A(\xi)$ be given by \eqref{eq:A} and suppose that $ V_0 \in C^\infty ( \mathbb R^n ) $, $ \min V_0 = 0 $,
  $ V_0 ( x ) > c_0 > 0 $ for $  |x| \geq R $.  Then, there exists
  $ C_0 $ and $ h_0 $ such that for $ 0 < h < h_0 $, 
  \begin{equation*}
    \Spec ( A(hD) + V_0  )  \cap [ 0 , C_0 h ]  =
  \Spec_{\rm{pp}} ( A(hD) + V_0  )  \cap [ 0 , C_0 h ] . \end{equation*}
  If $ \lambda (h) := \min \Spec_{\rm{pp}} ( A(hD) + V_0 ) $ and
  $ ( A(hD) + V_0  ) u = \lambda ( h ) u $, $ \| u \| = 1 $, then
  \begin{equation}
    \label{eq:der10}
 h^2/C_0 \leq \lambda ( h ) \leq C_0 h  \ \ \text{ and }    \ \ 
  \| ( h D)^2 u \|^2 \leq C h \lambda ( h ) .
  \end{equation}
  \end{lemm}
\begin{proof}
We follow the steps in the proof of Lemma \ref{l:ground}. As in the beginning of the proof of
Lemma \ref{l:local} we see that there exists $ \chi \in C^\infty_{\rm{c}} ( \mathbb R^{2n} ) $
such that $ A ( hD ) + V_0 ( x ) + \chi^{\rm{w}} ( x, h D ) - \lambda $, $ \lambda = 
\mathcal O ( h ) $,  is invertible. We then 
have an analogue of \eqref{eq:resol} (with $ - h^2 \Delta $ replaced by $ A ( h D) $ and 
$ V_1 $ by $ V_0 + \chi^{\rm{w}} ( x, h D ) $) which shows that the resolvent of 
$ A ( hD ) + V_0 $ is meromorphic near $ 0 $ and the spectrum is discrete there. 

To see the upper bound on $ \lambda ( h ) $, 
 suppose that $ \psi \in C^\infty_{\rm{c}} ( B ( 0 ,1 ) ) $ and $ \int_{\mathbb R^n } | \psi ( x ) |^2 dx = 1. $
Suppose that $ V_0 ( x_0) = 0 = \min V_0 $ and assume without loss of generality that $ x_0 = 0$.  Then $ V_0 ( x ) = \mathcal O ( |x|^2 ) $. We define $B(t)$ so that
$2(2-\cos t-\psi_0(t))=t^2B(t)^2$ and thus
\begin{equation}
\label{eq:A2B}
\begin{gathered}
A(hD)=\sum_{j=1}^n B_j(hD)^2 (hD_{x_j})^2,\quad B_j ( \xi ) = B ( \xi_j ) , \\
B_j \in S ( \langle \xi_j \rangle^{-1} ), \ \ B_j (\xi ) \geq \langle \xi_j \rangle^{-1} / C , \ \ \xi \in \mathbb R^n.
\end{gathered}
\end{equation}
For $ \psi_h ( x) := h^{-\frac n4} \psi ( h^{-\frac12} x ) $ we have $ 
( A(hD) + V_0 ) \psi_h    = h \widetilde \psi_h $, where 
\[  \begin{split}  \widetilde \psi_h ( x )   & = h^{-\frac n4} \sum_{j=1}^n B_j(hD)^2 (D_j^2 \psi) (h^{-\frac12} x ) + 
h^{-\frac n4} \mathcal O ( ( |h^{-\frac12} x |^2  ) \psi ( h^{-\frac 12 }  x  ) = \mathcal O ( 1 )_{L^2} 
 .\end{split}  \]
This shows that $ \lambda ( h ) \leq C_0 h $. 

For the lower bound on $ \lambda ( h ) $,  we replace 
$ u $ by a quasimode localized using $ \chi \in C^\infty_{\rm{c}} $ (allowed by Lemma \ref{l:loc})
to see that a pairing of \eqref{eq:A2B} with 
$ u $ gives
\begin{equation}
\label{eq:B2la}  \sum_{ j=1}^n \| B_j ( h D) h D_{x_j} u \|^2 + \| V_0^{\frac12} u \|^2 = \lambda ( h ) + \mathcal O ( h^\infty ) .
\end{equation}
Since $ u $ is localized in $ \xi $ and $ x $, and $ B_j ( hD ) $ is elliptic,  we can again 
use the uncertainty principle: 
\[  \| B_j ( h D ) h D_{x_j}  u \| \geq \| h D_{x_j } u \|/C - \mathcal O ( h^\infty ) \geq
h/ ( 2 C \| x_j u \|) \geq h/C_1. \]
Combined with \eqref{eq:B2la} this gives $ \lambda ( h ) \geq h^2/C_1^2 $. 

To obtain the second part of \eqref{eq:der10} we again differentiate the equation with $ h D_{x_k}$
and pair the result with $ h D_{x_k } u $ (see \eqref{eq:pair})
\[  \sum_{ j=1}^n \| B_j ( h D ) hD_{x_j } D_{x_k } u \|^2 \leq \lambda ( h ) \| hD_{x_k} u \|^2 +
 C h \| V_0^{\frac12} u \| \| h D_{x_k } u \| \leq C _1 h \lambda( h) . \]
 Since $ u $ is localized (by Lemma \ref{l:loc}) and $ B_j $'s are elliptic (in $ S ( \langle \xi_j \rangle^{-1} ) $), 
 we obtain the desired bound.
\end{proof}
    
Suppose $ \varphi \in C_{\rm{c}}^\infty ( \mathbb R^n; [0,1]) $ is supported in a small neighbourhood of $ 0 $ and equal to $ 1 $ near $ 0 $. Then $ u $ in the statement of Lemma \ref{l:ground}
satisfies
\[   u = \varphi ( h D ) u + \mathcal O ( h^\infty )_{ \mathscr S} . \]
(This follows from semiclassical ellipticity of $ - h^2 \Delta + V_0 - \lambda ( h) $ for $ |\xi| > \delta $ for any $ \delta > 0$, 
see \cite[Theorem 4.29]{z12}.) We make sure to choose $ \varphi$ so that $\psi_0\equiv1$ on $\supp\varphi$, where $\psi_0$ is given by \eqref{eq:A}, and write 
\[  \sum_{ j=1}^n 2 ( 1 - \cos \xi_j ) \varphi ( \xi ) = |\xi|^2 \varphi ( \xi ) + 
\sum_j^n \xi_j^4 a_j (\xi ) , \ \ \   a_j \in C^\infty_{\rm{c}} ( \mathbb R^n )   . \]
Then
\[ \begin{split}  \langle P_1 u , u \rangle & = 
\langle  ( A ( h D) \varphi ( h D ) + V_0 )u, u \rangle + \mathcal O ( h^\infty ) \\
& =  \langle   ( - h^2 \Delta + V_0 ) u , u \rangle + \sum_{ j=1}^n \langle a_j ( h D ) ( hD_{x_j}) ^4 u, u \rangle
+ \mathcal O ( h^\infty ) \\
& \leq \lambda (h) + C \| (h D)^2 u \|^2 + \mathcal O ( h^\infty ) \leq \lambda (h) ( 1 + C_1 h ) , \end{split} \]
where we used \eqref{eq:der} to get the last inequality. This shows that
\begin{equation}
\label{eq:minH0}  \min \Spec ( H_0 ) = \min \Spec ( P_1 ) + \mathcal O ( h^\infty ) \leq 
\min \Spec ( P_0 ) ( 1 + \mathcal O ( h ) ) . \end{equation}
Similarly we obtain the inequality with $ H_0 $ and $ P_0 $ replaced (using Lemma \ref{l:groundP1} instead of 
Lemma \ref{l:ground}):
\begin{equation}
\label{eq:minP0}
 \min \Spec ( P_0 )  \leq 
\min \Spec ( P_1 ) ( 1 + \mathcal O ( h ) ) = \min \Spec ( H_0 ) ( 1 + \mathcal O ( h ) )  
 . \end{equation}

The inclusions \eqref{eq:H2H0} show that
\[   \min \Spec ( H ) \geq \min \Spec (H_0 ) - \mathcal O ( h^\infty ) , \ \ \
\min \Spec ( P ) \geq \min \Spec ( P_0 ) - \mathcal O ( h^\infty) . \]
Since we can use the ground state eigenfunctions of $ H_0 $ and  $ P_0 $ as test functions for
$ H $ and $ P $ respectively (their localization to $ \Omega $ is key here), the opposite inequalities 
follow.  Combined with \eqref{eq:minH0} and \eqref{eq:minP0} this concludes the proof of \eqref{eq:comp}.

\section{Numerical experiments} 

Numerical investigation of the spectrum of the discrete operator \eqref{eq:Pd} poses some subtle challenges
which require review of different aspects of the theory. Our experiments are conducted in 1D and we specialize
to that case in this section.

\subsection{Auxiliary family of operators}
Let $V \in C^1(\mathbb T^1)$, 
$ \mathbb T^1 := \mathbb R /  \mathbb Z $. 
 To study $\Spec(P_{\rm d}(h))$ for incommensurable $h$, we introduce an auxiliary family generalizing \eqref{eq:Pd}
\begin{equation*}
 (P_{\rm d}(h,\theta)v)(\gamma) := 2v(\gamma)-v(\gamma+1) - v(\gamma-1) + V( h\gamma +\theta)v(\gamma), \quad \gamma \in\ZZ
 \end{equation*}
with $h>0$ and $\theta \in \mathbb R / \ZZ.$ In addition, we define 
\begin{equation}
\label{eq:defSig} \Sigma_h := \bigcup_{\theta \in \mathbb R / \ZZ} \Spec(P_{\rm d}(h,\theta)).
\end{equation}
We recall the following well-known result.
\begin{lemm}\cite[(1.2)]{HS88}
\label{lemm:easy_obsv}
For $h \notin  \mathbb Q$ the spectrum of $P(h,\theta)$ does not depend on $\theta$ and thus $\Sigma_h = \Spec(P(h,\theta))$ for all $\theta \in \mathbb R / \ZZ.$
\end{lemm}

Under the assumption $h =p/q \in \mathbb Q$, operators $P_{\rm d}(h,\theta)$ are periodic and we can describe the spectrum of $P_{\rm d}(h,\theta)$ using Bloch-Floquet theory with Bloch boundary condition $v_{n+q}=e^{i \theta_1 q} v_n.$ This implies that (see \cite[Remark 1.10]{KRL14})
\[ \Spec(P_{\rm d}(h,\theta'))= \bigcup_{\theta_1 \in \mathbb R / \ZZ} \Spec(M_{(\theta_1,\theta')})  \ \text{ and }   \ \Sigma_{h}  = \bigcup_{\theta \in \mathbb R^2 / \ZZ^2 } \Spec(M_{\theta})\]
with $M_{\theta} \in \CC^{q\times q}$ given by 
\begin{gather*} M_{\theta}: = 2 I_q - e^{- i \theta_1} K_{q}^* - e^{ i \theta_1} K_{q}+ \sum_{\beta \in \ZZ} w_{\beta} e^{i\beta \theta_2} (J_{p,q})^{\beta}, \\
J_{p,q} := \operatorname{diag}(e^{2\pi i (j-1)p/q}),  \ \ \ \ (K_q)_{ij}: =\begin{cases} 1 \text{ if } j=(i+1) \!\! \!\!\mod  \! q \\ 0 \text{ otherwise.} 
\end{cases}
\end{gather*}

It follows that approximating the spectrum of $\Spec(P_{\rm d}(h))$ for $h \in \mathbb Q$ reduces to computing spectra of $M_{\theta_1,0}$ over a fine enough discretisation of $\theta_1.$ 
To study $\Spec(P_{\rm d}(h))$ for $h \in \mathbb R \setminus \mathbb Q$, we use that $\Sigma_h = \Spec(P_{\rm d}(h))$ by Lemma \ref{lemm:easy_obsv} and the following quantitative continuity bound
 obtained by Avron, v.~Mouche and Simon:
\begin{lemm}\cite[Proposition 7.1]{AMS78}
\label{lemm:easy_obsv2} Let $V \in C^1(\mathbb R / \ZZ)$, then the map $\mathbb R \ni h \mapsto \Sigma_h$ is $1/2$-H\"older continuous in Hausdorff distance $d_H$. In particular, for $h,h' \in \mathbb R$
\[ d_H(\Sigma_{h},\Sigma_{h'}) \le C_V \vert h - h'\vert^{1/2},\]
\end{lemm}
This results implies that the spectrum $\Spec(P_{\rm d}(h))$ for irrational $h$ can be well-approximated by computing spectra of $M_{\theta}$ for $\theta \in \mathbb T^2.$  Thus, in addition to $ d( h) $ defined in 
\eqref{eq:comp}, we numerically study
\begin{equation*}
D ( h ) := \frac{ \min \Sigma_h}{ \min \Spec ( P_{\rm{c}} ( h ) ) },  \end{equation*} 
where $ \Sigma_h $ was given in \eqref{eq:defSig}.

The numerical results are shown in Figures \ref{fig:rational} and \ref{f:optimal} for rational values of 
$ h$ and in Figures \ref{fig:irrational} and \ref{f:optimal2} for irrational values (that is, using $ D( h )$ above as proxy). The figures show the optimality of Theorem \ref{thm:main} in both cases. For irrational $h$ we have $D(h)=d(h)$. For rational $h$ we have $D(h)\le d(h)$. While $d(h)$ in general is not continuous in $h$, $D(h)$ is continuous in $h$. 

The experiments are performed for four potential:
\begin{enumerate}\label{eq:fourpotentials}
\item $V_1(x):=1+\cos(2\pi x)$, the \emph{subcritical} Harper operator with non-degenerate minima.
\item $V_2(x):=\cos^4(2\pi x)$, a potential with degenerate minima, that is $V^{(k)}(1/4)=0,$ 
$ k < 4 $, but $V^{(4)}(1/4) = 24(2\pi)^4.$
\item $V_3(x) := \frac12 {\sin(4\pi x)} - \cos(2\pi x)  + \frac14 {3\sqrt{3}} $, a potential with asymmetric wells at points $x =  \ZZ - \frac{1}{12}$ with series expansion $V_3(x)= {3\sqrt{3}}\pi^2 (x+\frac1{12})^2 - 2\pi^3 (x+ \frac{1}{12})^3 - 3\sqrt 3\pi^4 (x+ \frac{1}{12})^4+ \mathcal O((x+ \frac{1}{12})^5).$
\item $V_4(x):=\operatorname{exp}(-1/\sin^2(2\pi x))$, a non-analytic $C^{\infty}$ potential with minima of infinite order at the zeros of the sine function.
\end{enumerate}

We now show that it is in general \emph{not true} that $\mathbb R \ni h \mapsto \min \Spec(P_{\rm d}(h))$, let alone $\mathbb R \ni h \mapsto \Spec(P_{\rm d}(h)),$ is continuous. The maps however are continuous on $ \mathbb R \setminus \mathbb Q.$


\begin{figure}
{\begin{tikzpicture}
\node at (-6,0) {\includegraphics[width=16cm]{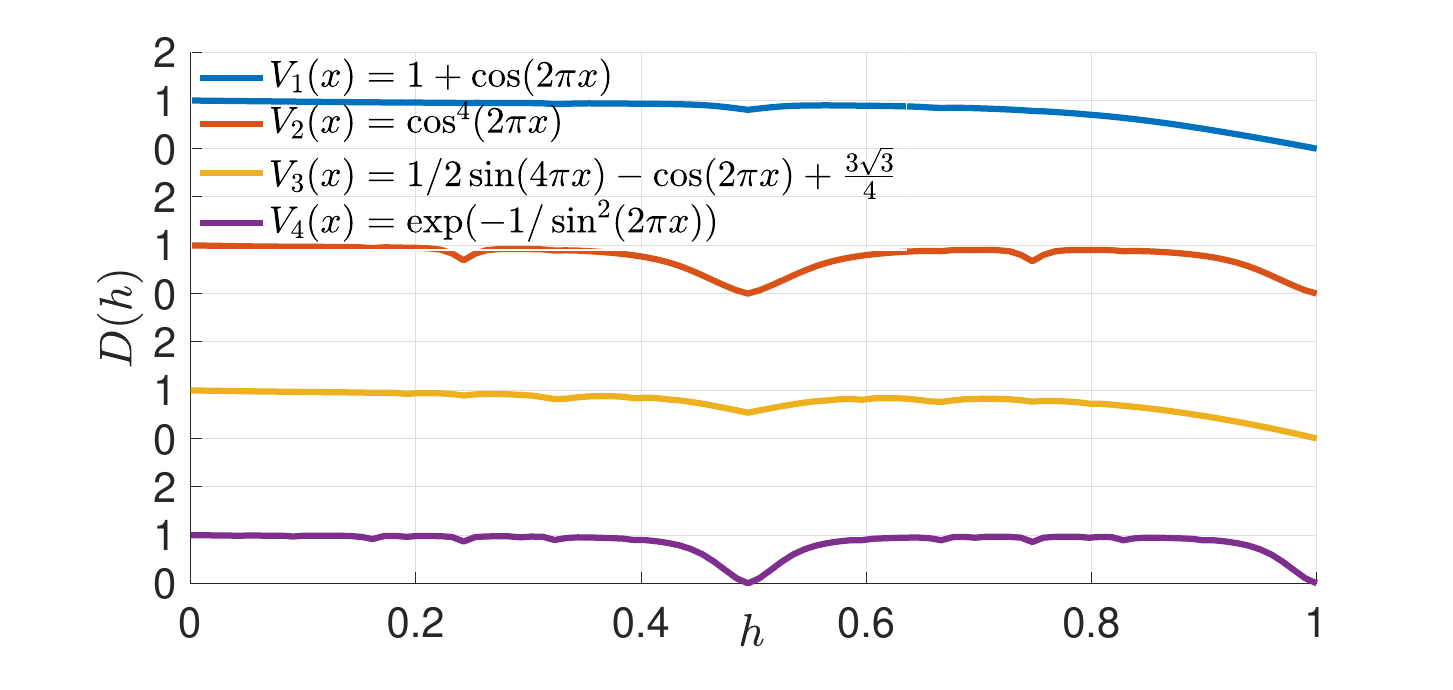}};
\end{tikzpicture}}
\caption{\label{fig:irrational}Illustration of the limiting quantity $D ( h ) := \frac{ \min \Sigma_h}{ \min \Spec ( P_{\rm{c}} ( h ) ) } $ in \eqref{eq:comp} sampled at rational $h \in  \mathbb Q$  and stacked vertically. Since $D(h)$ is continuous in $h$ and $D(h)=d(h)$ for irrational $h$ this serves as a proxy for the limiting quantity in \eqref{eq:comp} for $h\in\mathbb R\setminus\mathbb Q$. It follows from \cite{DSS} that
$ D( h ) \leq 1 $ for $ h \in \mathbb R $.}
\end{figure}

\begin{figure}
{\begin{tikzpicture}
\node at (-6,0) {\includegraphics[width=16cm]{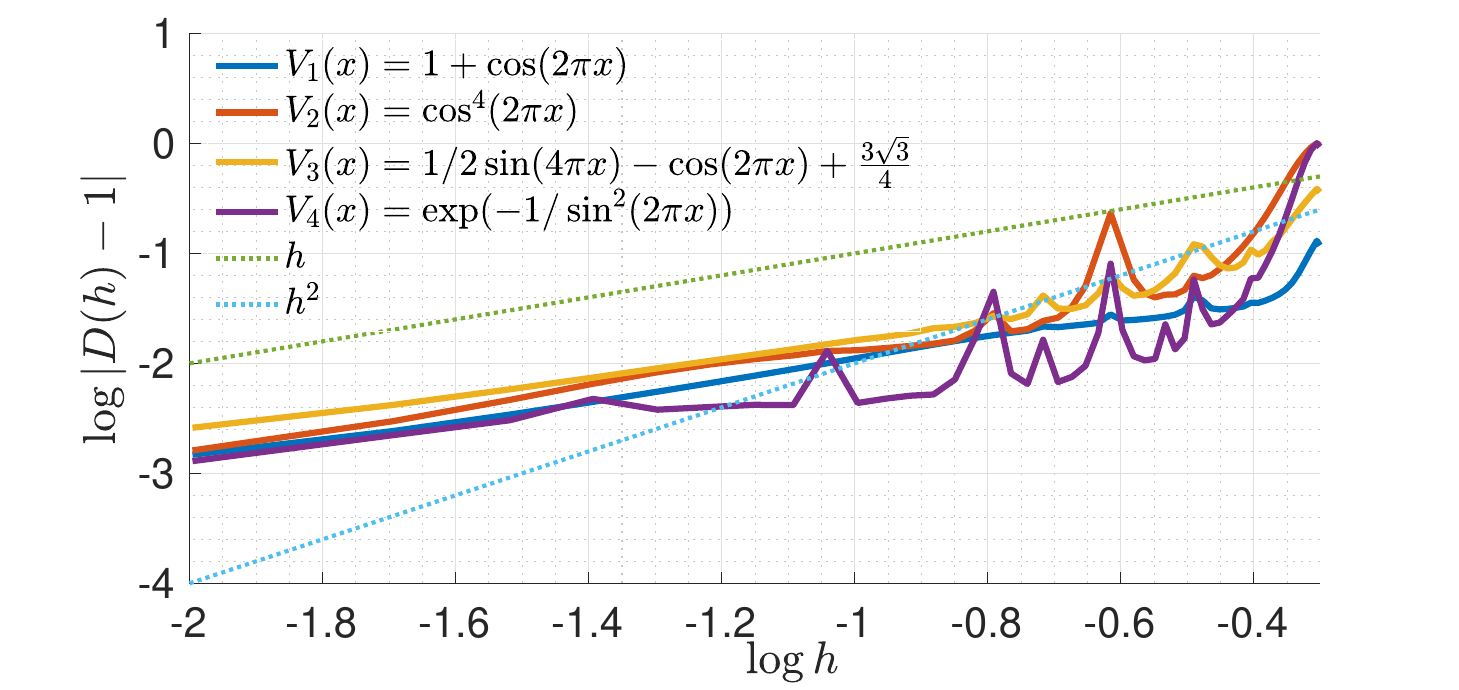}};
     \end{tikzpicture}}
\caption{\label{f:optimal2} $\log\log $ plot associated with Figure \ref{fig:irrational} indicating that $ |D( h )-1| = \mathcal O (h ) $ in Theorem \ref{thm:main} 
is optimal.}
\end{figure}

\subsection{(Dis)-continuity of spectra}
\label{s:con}
To explain the discrepancy between Figures \ref{fig:rational} and \ref{fig:irrational}, especially close to $h =1/2$ and $h =1$, we argue that $\Spec ( P_{\rm{d}} ( h ) )$ and also $\min \Spec ( P_{\rm{d}} ( h ) )$ are in general discontinuous in the parameter $h \in \mathbb R$. By Lemma \ref{lemm:easy_obsv2}, it suffices to show that $\Spec(P_{\rm d}(h))  \neq \Sigma_h$ and $\min \Spec(P_{\rm d}(h)) >\min \Sigma_{h}$, respectively for suitable choices of $h$. We illustrate this in Figures \ref{fig:Hofstadter1} and \ref{fig:Hofstadter2}, respectively. 

\subsubsection{Discontinuities at $h =\frac12$}
\label{s:dis}

For $V(x) = 1+\cos(2\pi x)$ we find
\[M_{\theta} = \begin{pmatrix} 3-2\cos \theta_1 & \cos \theta_2  \\   \cos \theta_2  & 3-2\cos(\theta_1 + \pi) \end{pmatrix}. \]
The spectrum of this matrix is given by 
\[ \Spec(M_{\theta}) = \left\{ 3\pm \frac{\sqrt{5 + 4\cos(2 \theta_1) + \cos(2\theta_2)}}{\sqrt{2}} \right\}.\]
For $\theta_2 =0$, we have $\bigcup_{\theta_1 \in \mathbb R} \Spec(M_{\theta_1,0}) = [3-\sqrt{5},2] \cup [4,3+\sqrt{5}],$ 
while $\bigcup_{\theta \in \mathbb R^2} \Spec(M_{\theta}) = [3-\sqrt{5},3+\sqrt{5}].$

Similarly, if $V(x)=\cos^4(2\pi x)$, then 
\[M_{\theta} = \operatorname{diag}(2-2\cos(\theta_1) + \cos^4(\theta_2),2-2\cos(\theta_1+\pi) + \cos^4(\theta_2)),\]
with $\bigcup_{\theta_1 \in \mathbb R} \Spec(M_{\theta_1}) = [1,5]$ while $\bigcup_{\theta \in \mathbb R^2} \Spec(M_{\theta}) = [0,5].$

\subsubsection{Discontinuities at $h=1$}
For $h =1$ and $V(x) = 1+\cos(2\pi x)$, we have
\[ M_{\theta} = 3-2 \cos \theta_1  + \cos \theta_2 \]
which implies that $\bigcup_{\theta_1 \in \mathbb R}  \Spec(M_{\theta_1,0}) = [2,6]$ while $\bigcup_{\theta \in \mathbb R^2} M_{\theta} = [0,6].$

For $V(x)=\cos^4(2\pi x)$ we have 
\[ M_{\theta} = 2-2 \cos \theta_1  + \cos^4 \theta_2 \]
which implies that $\bigcup_{\theta_1 \in \mathbb R}  \Spec(M_{\theta_1,0}) = [1,5]$ while $\bigcup_{\theta \in \mathbb R^2} M_{\theta} = [0,5].$

\begin{figure}
{\begin{tikzpicture}
\node at (-6,0) {\includegraphics[width=14cm]{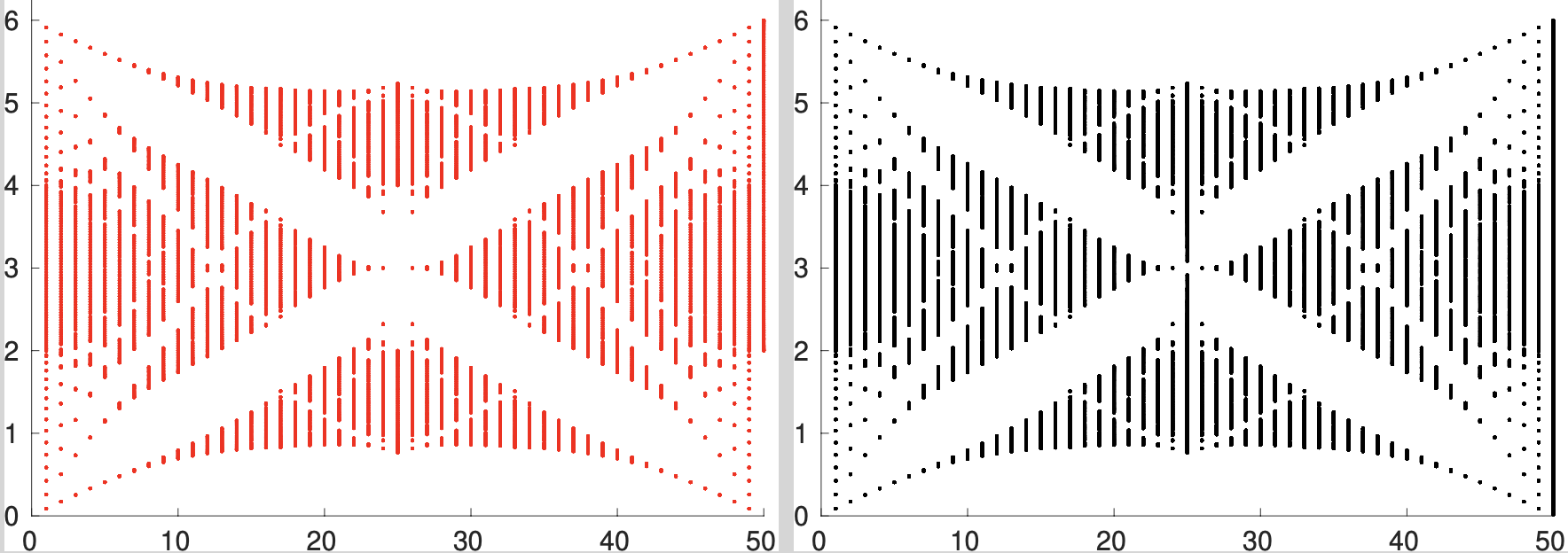}};
\draw [-{Stealth[length=3mm, width=2mm]}](-0.8+1,2.7)--(-0.8-1.7,0.05);
\draw [-{Stealth[length=3mm, width=2mm]}] (-0.8+1,2.7) -- (-0.8+1.7,-2.0);
\draw [-{Stealth[length=3mm, width=2mm]}](-0.8-6,2.7)--(-0.8-1.7-7,0.05);
\draw [-{Stealth[length=3mm, width=2mm]}] (-0.8-6,2.7) -- (-0.8+1.7-7,-2.0);
     \node at  (-9.5-3.5,2.7){$\Spec(P_{\rm d}(h)) $};
     \node at (-3.5+1-3.5,2.7){$\Sigma_{h} $};
     \node at  (-9.5,2.7){$V(x) = 1+\cos(2\pi x)$};
     \node at (-3.5+1,2.7){$V(x) = 1+\cos(2\pi x) $};
     \node at (-3.5+1,-2.7){$50 h $};
     \node at (-9.5,-2.7){$50 h $};
\end{tikzpicture}}
\caption{\label{fig:Hofstadter1} Hofstadter butterfly $h= \frac{p}{50} \mapsto \Spec(P_{\rm d}(h)) $ (left) and $h= \frac{p}{50} \mapsto \Sigma_{h}$ (right) for $V(x)=1+\cos(2\pi x)$. Discontinuities at $p=25,50$ are clearly visible on the left. Visible discrepancies between the two butterflies are indicated by arrows. }
\end{figure}

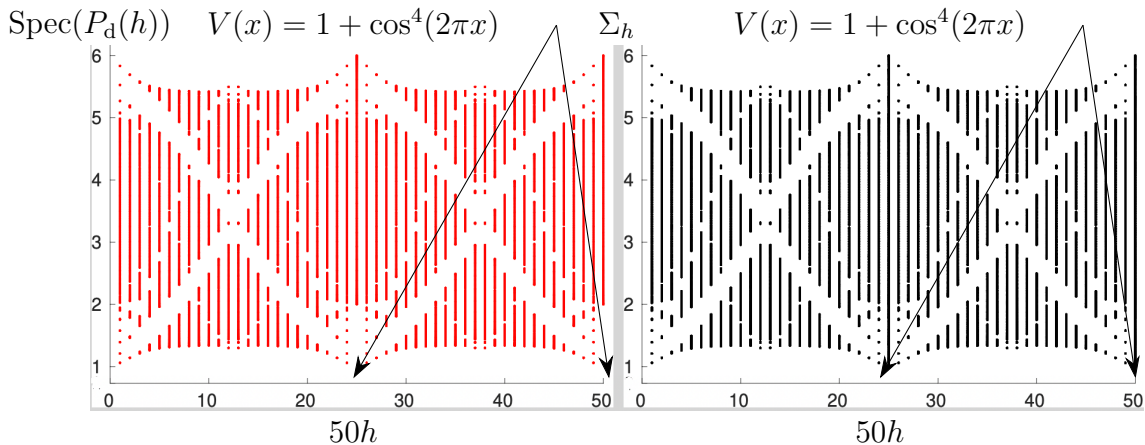
\begin{figure}
{\begin{tikzpicture}
\node at (-6,0) {\includegraphics[width=14cm]{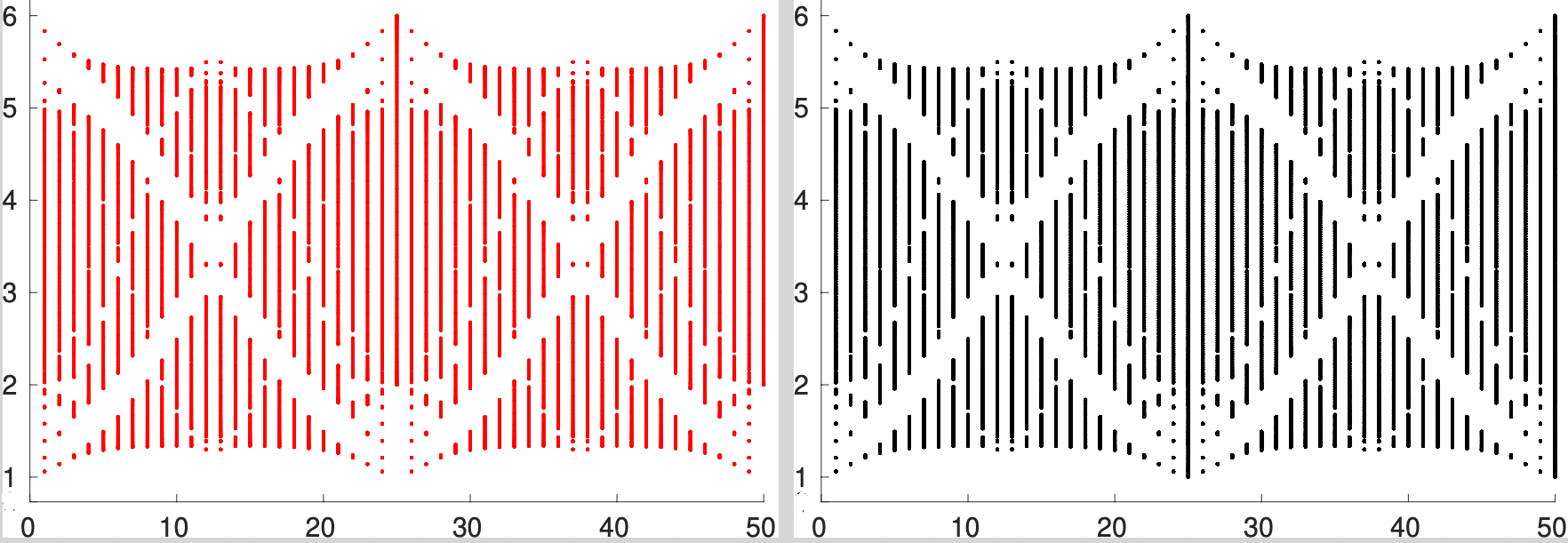}};
\draw [-{Stealth[length=3mm, width=2mm]}](-0.8+1,2.7)--(-0.8-1.7,-2.0);
\draw [-{Stealth[length=3mm, width=2mm]}] (-0.8+1,2.7) -- (-0.8+1.7,-2.0);
\draw [-{Stealth[length=3mm, width=2mm]}](-0.8-6,2.7)--(-0.8-1.7-7,-2.0);
\draw [-{Stealth[length=3mm, width=2mm]}] (-0.8-6,2.7) -- (-0.8+1.7-7,-2.0);
     \node at  (-9.5-3.5,2.7){$\Spec(P_{\rm d}(h)) $};
     \node at (-3.5+1-3.5,2.7){$\Sigma_{h} $};
     \node at  (-9.5,2.7){$V(x) = 1+\cos^4(2\pi x)$};
     \node at (-3.5+1,2.7){$V(x) = 1+\cos^4(2\pi x) $};
     \node at (-3.5+1,-2.7){$50 h $};
     \node at (-9.5,-2.7){$50 h$};
\end{tikzpicture}}
\caption{\label{fig:Hofstadter2} Hofstadter butterfly $h = \frac{p}{50} \mapsto \Spec(P_{\rm d}(h)) $ (left) and $h= \frac{p}{50} \mapsto \Sigma_{h}$ (right) for $V(x)=1+\cos^4(2\pi x)$. Discontinuities at $p=25,50$ are clearly visible on the left. Visible discrepancies between the two butterflies are indicated by arrows. }
\end{figure}

\section{The Bohr-Sommerfeld quantization condition}
\label{s:BS}
\begin{figure}
\includegraphics[width=13cm]{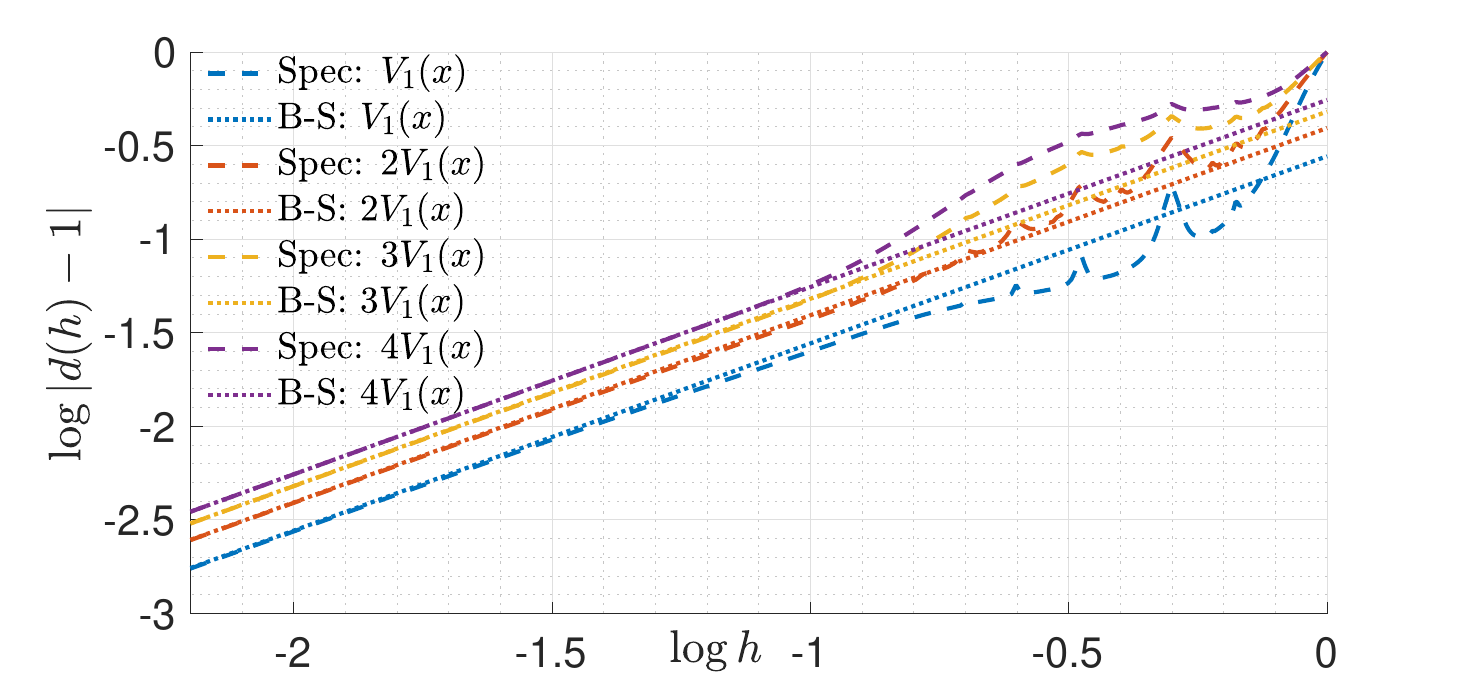}
\includegraphics[width=13cm]{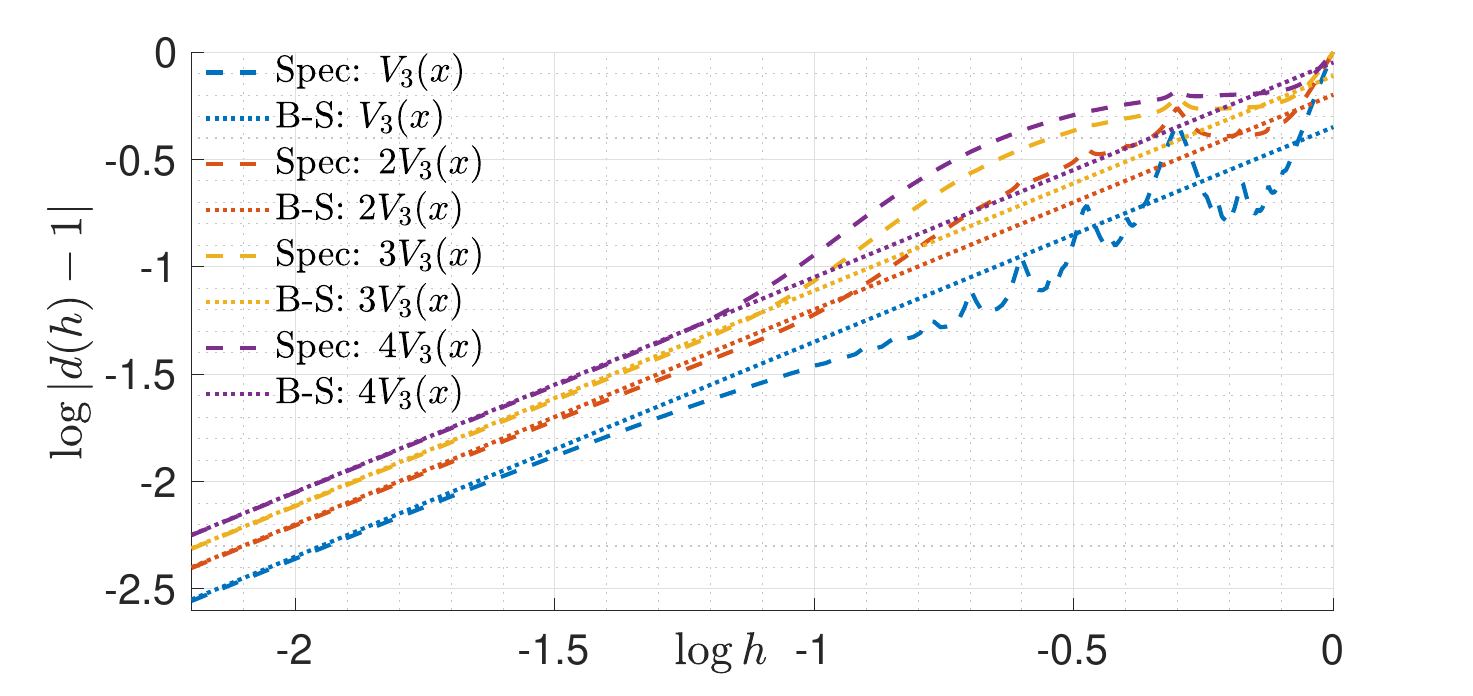}
\caption{\label{fig:dh} Linear convergence rate of $\vert d(h)-1\vert$ illustrated by comparison with leading term obtained from Bohr-Sommerfeld quantization condition in \eqref{eq:BSdh} to spectral computations for various multiples of the potentials $V_1$ (top) and $V_3$ (bottom), i.e. operators $p^{\rm w}(x,hD)=-h^2\Delta+\lambda V_j$ and $p^{\rm w}(x,hD)=2(1-\cos(hD))+\lambda V_{j}(x)$ and $\lambda \in \{1,2,3,4\}.$ }
\end{figure}
The distribution of eigenvalues of a symbol exhibiting a potential well can be expressed using a Bohr-Sommerfeld quantization condition. 
A rigorous approach to that condition 
was developed by Helffer-Robert \cite{HR84}. It allows us to compare our numerical computations of the bottom of the spectrum to the analytic expressions obtained from the quantization condition. We let $p^{\rm w}(x,hD)$ be either 
\begin{equation}\label{eq:comparisonops}
p^{\rm w}(x,hD)=-h^2\Delta+V_{0}(x)\quad \text{ or }\quad p^{\rm w}(x,hD)=2(1-\cos(hD))+V_{0}(x)
\end{equation}
(see~\eqref{eq:P1}) with $V_0$ as in \eqref{eq:V0}. We use an elegant presentation of higher order Bohr--Sommerfeld rules developed by Colin de Verdi\`ere \cite{cdv},
\begin{equation}
\label{eq:BS} 
2\pi (n+\tfrac{1}{2})h=S(E_n,h)=\sum_{j=0}^\infty h^{2j}S_{2j}(E_n),\quad n\in\mathbb N_0.
\end{equation}
It follows from \cite{HS88} (see also \cite{hz} for a recent treatment of a more general case) that for analytic potentials $V$ with non-degenerate wells, the Bohr--Sommerfeld rule is also valid for the bottom of the spectrum, $n=0$. The analyticity assumption is not essential: asymptotic formulas for the ground
state in the case of non degenerate minimum hold in all dimensions -- see \cite[Theorem 3.6, Theorem 4.23]{DiSj} and references given there. In the analytic case we can use the more computationally tractable Bohr--Sommerfeld rules but the expansions in terms of the potential are ultimately the same.

Hence, we restrict ourselves to the two potentials $V_1,V_3$ described on page \pageref{eq:fourpotentials}, as the quantization rule may not apply to $V_2$ (degenerate well) and $V_4$ (non-analytic potential with an infinitely degenerate well). 

The first two terms in \eqref{eq:BS} are given by 
\[ S_0(E)=\int_{\{p\le E\}} dx \,d \xi , \quad  S_2(E)=-\tfrac{1}{24}\partial_E^2\int_{\{p\le E\}}\operatorname{det}(p'')\,dx\,d\xi  . \]
The first term is the classical action ($ dx\, d\xi $ is the Lebesgue measure on $ \mathbb R^2 $). The 
second one comes from 
the proof of \cite[Theorem 2]{cdv}. For analytic $V_0$ such that 
\begin{equation}\label{eq:V}
V_0(x)=a_0x^2+a_1 x^3 + a_2x^4 + \mathcal O(x^5)
\end{equation}
near $x=0$ with $a_0>0$ (which holds for appropriate translations of $V_1,V_3$ with $a_0=2\pi^2$, $a_1=0$, $a_2=-2\pi^4/3$ and $a_0=3\sqrt 3 \pi^2$, $a_1=-2\pi^3$, $a_2=-3\sqrt 3 \pi^4$, respectively) we have 
\begin{equation}\label{eq:S0}
 S_0 ( E ) = \pi a_0^{-1/2}E + \pi \alpha_1 E^2  + \mathcal O(E^3), \ \  S_2 ( E ) = \pi \alpha_2  + \mathcal O( E )  
 \end{equation}
for $\alpha_i$ specified in Table \ref{table:BS} for $V_1,V_3$ and $p$ in \eqref{eq:comparisonops}. Here, we used that the leading term of $S_0(E)$ is given only in terms of the leading order Taylor coefficient of the quadratic well potential. This fact and further details on how to obtain the coefficients can be found in the appendix.
From the expansion of the action \eqref{eq:S0} and the quantization condition \eqref{eq:BS}, we immediately obtain
\begin{equation}
\label{eq:E0mod}  E_0 ( h ) = a_0^{1/2} h -  a_0^{1/2} (a_0 \alpha_1 + \alpha_2 ) h^2 + \mathcal O(h^3 ). 
\end{equation}
As seen in Table \ref{table:BS}, there is a difference in $\alpha_i$ depending on the choice of $p$ in \eqref{eq:comparisonops}. This difference only depends on the leading term in \eqref{eq:V}, and this is true for any such potential, see Lemmas \ref{lem:S0} and \ref{lem:S2} in the appendix. We can relate this to $d(h)$ in \eqref{eq:comp} by inserting the semiclassical approximation for \eqref{eq:E0mod} which gives (see the remarks after \eqref{eq:BS})
\begin{prop}
Suppose that  $V \in C^\infty ( \mathbb R; \mathbb R ) $ is $\mathbb Z$-periodic, 
$ V $ has a single minimum, $x_0 $, in a fundamental domain and $V(x)=a_0(x-x_0)^2 + \mathcal O(\vert x-x_0\vert^3),$  $a_0 >0$. Then
\begin{equation}
\label{eq:BSdh}
d(h)= 1 - \frac{\sqrt{a_0}}{16} h + \mathcal O(h^2).
\end{equation}
\end{prop}
We illustrate the effectiveness of \eqref{eq:BSdh} in Figure \ref{fig:dh}. In Figures \ref{fig:BSc} and  \ref{fig:BSd}, we illustrate the fast convergence of the Bohr-Sommerfeld rule to compute $E_0(h)$ in \eqref{eq:E0mod}. In addition, in Figure \ref{fig:BSc}, we show asymptotics of $\min \Spec(P_{\rm c}(h))$ and $\min \Spec(P_{\rm d}(h))$ indicating the $h^{2k/(k+1)}$ behaviour,  where $2k$ is the order of vanishing of $ V $ at the bottom of the well. (This follows from a simple rescaling argument.) The scaling is visible for relatively large $h>0$ for analytic potentials $V_1,V_2,V_3$, but less so for the non-analytic potential $V_4.$

\begin{figure}
\begin{centering}
\includegraphics[width=7.6cm]{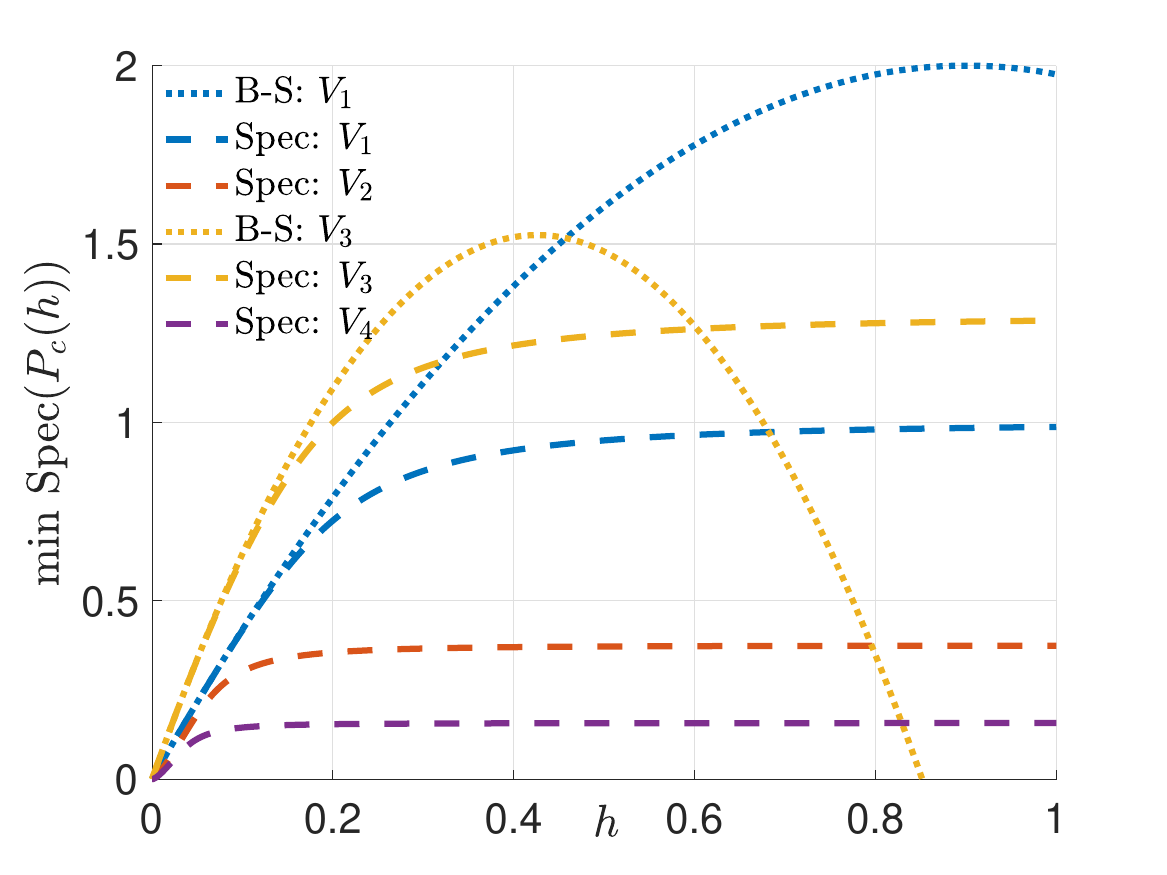}
\includegraphics[width=7.6cm]{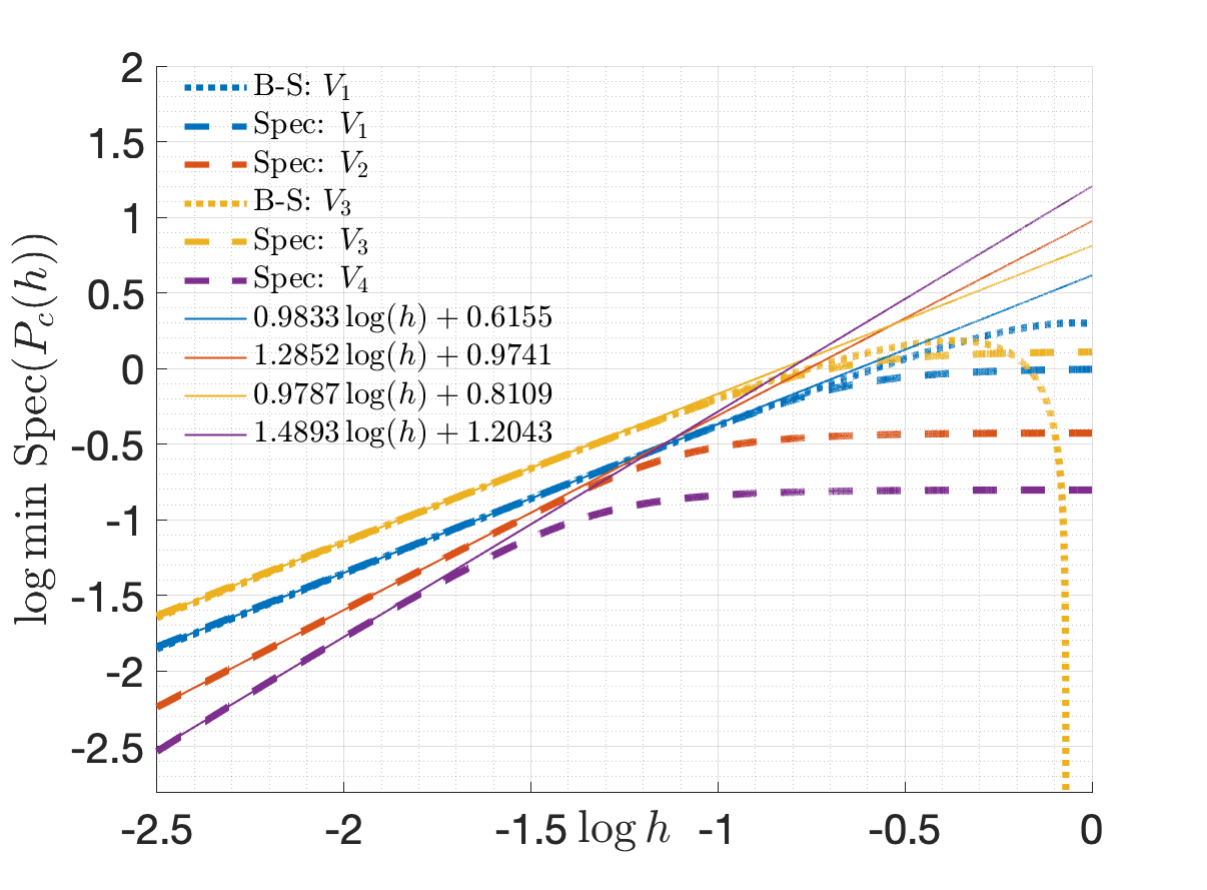}
\end{centering}
\caption{Comparison of $\min(\Spec(P_{\rm c}(h)))$ computed using direct spectral (Spec) computations and using the first two terms of the Bohr-Sommerfeld (B-S) condition, see \eqref{eq:E0mod}, for potentials $V_1,V_3$. 
We also include a linear fit at log-scale. It confirms the linear dependence on $ h $ 
for potentials with non-degenerate minima and suggests the expected $ h^{\frac43} $ behaviour
for the potential with the fourth order minimum and $ h^{\frac32} $ for $ V_4 $ (while a larger power is 
expected from scaling). 
\label{fig:BSc}}
\end{figure}

\begin{figure}
\begin{centering}
\includegraphics[width=7.6cm]{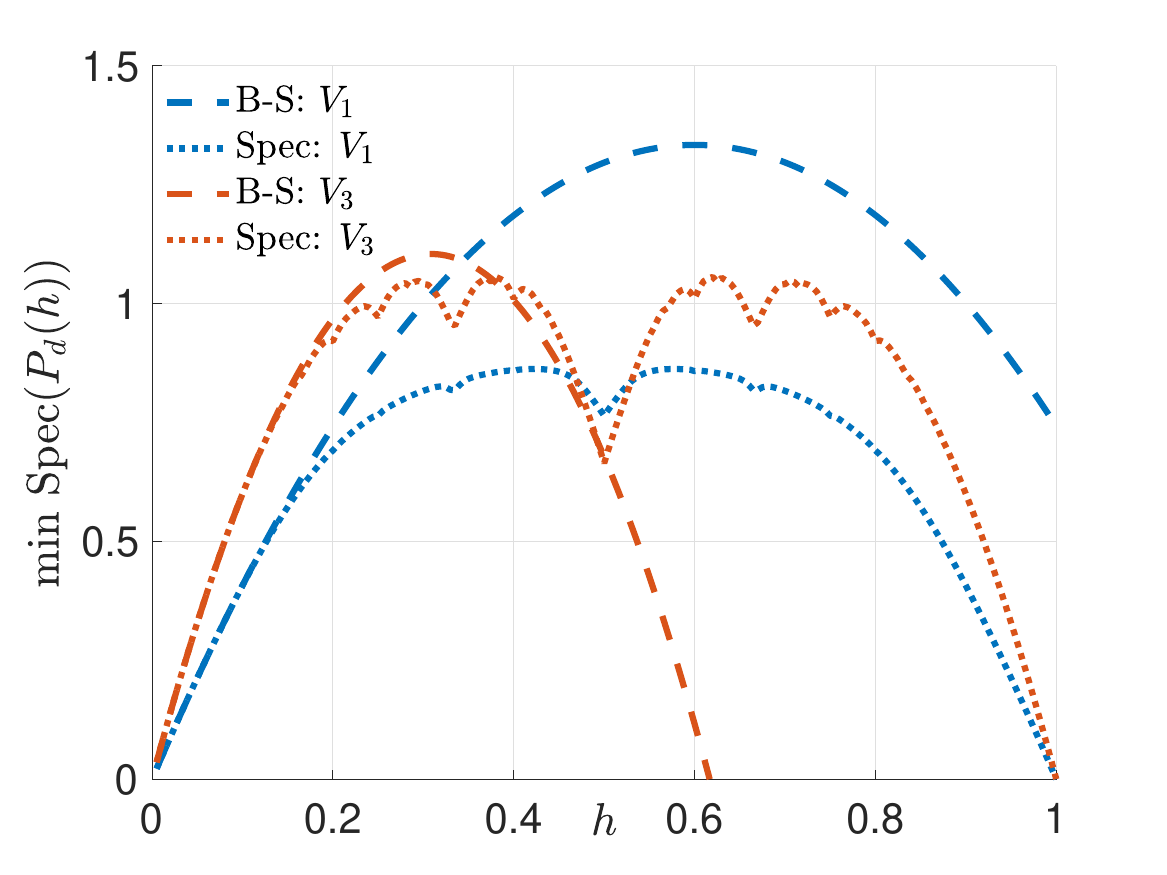}
\includegraphics[width=7.6cm]{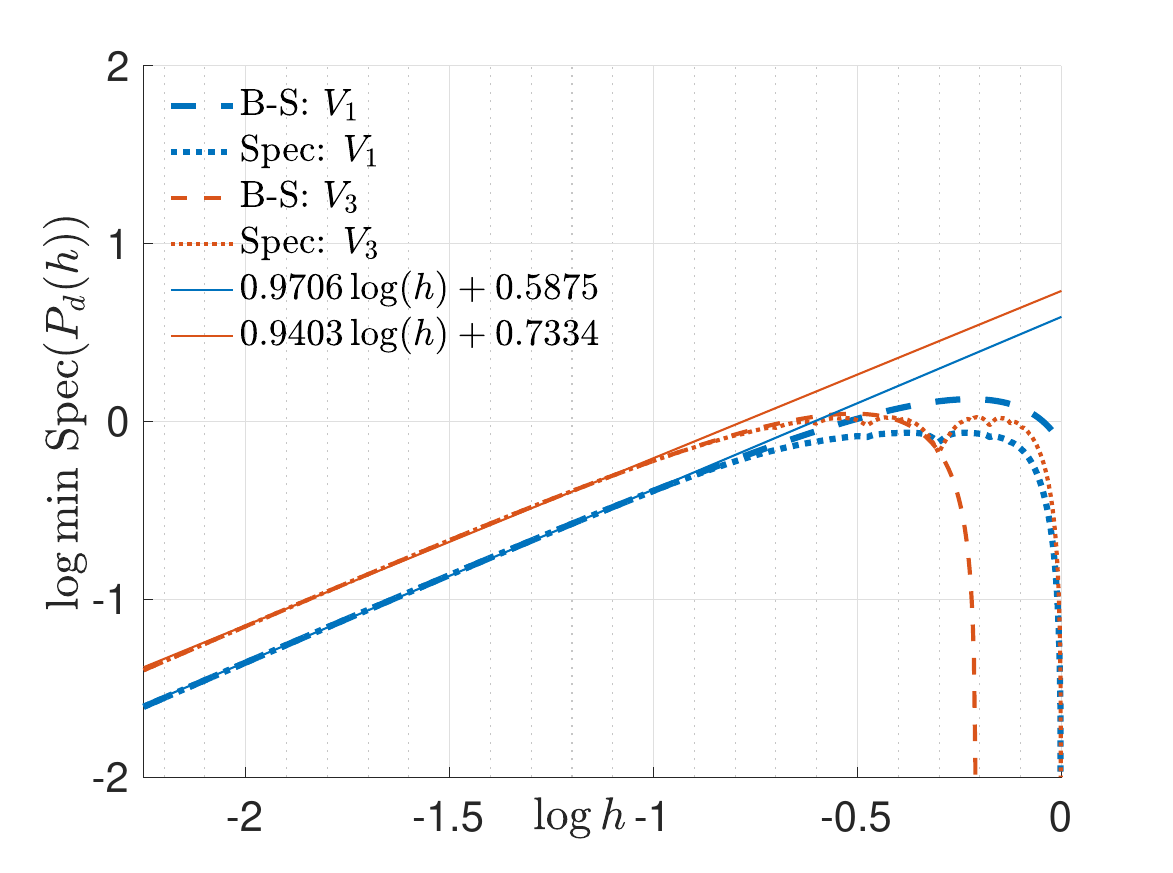}
\end{centering}
\caption{Comparison of $\min(\Spec(P_{\rm d}(h)))$ computed using direct spectral (Spec) computations and using the first two terms of the Bohr-Sommerfeld (B-S) condition, see \eqref{eq:E0mod}, for potentials $V_1,V_3$ with linear fit at log-scale. \label{fig:BSd}}
\end{figure}
\begin{center}
\begin{table}
\caption{\label{table:BS}Coefficients $\alpha_i$ appearing in the asymptotic expansion of the action terms $S_0, S_2$ in \eqref{eq:S0}.}
\begin{tabular}{ |p{1.5cm}|p{2cm}|p{2cm}| p{3cm}|p{3cm}| }
\hline
\multicolumn{5}{|c|}{Action coefficients} \\
\hline
\text{Symbol} & $\xi^2 + V_1$ & $\xi^2 + V_3 $ & $2(1-\cos(\xi)) + V_1$ & $2(1-\cos(\xi)) + V_3$ \\
\hline
$a_0$ & $2\pi^2$ & $3\sqrt 3 \pi^2$ & $2\pi^2$ & $3\sqrt 3\pi^2$ \\
$\alpha_1$ &  $\frac{1}{16\pi\sqrt{2}}$  & $\frac{4}{81\pi \cdot  3^{1/4}}$ & $\frac{1}{16\pi\sqrt{2}}+\frac{1}{32}a_0^{-1/2}$ & $\frac{4}{81 \pi\cdot  3^{1/4}} + \frac{1}{32}a_0^{-1/2}$ \\
$\alpha_2$  &$ \frac{\pi}{8 \sqrt{2}}$ & $\frac{11\pi}{27 \cdot 3^{3/4}}$ & $\frac{\pi}{8 \sqrt{2}}+\tfrac1{32}a_0^{1/2}$ & $\frac{11\pi}{27 \cdot 3^{3/4}}+\tfrac1{32}a_0^{1/2}$ \\
\hline
\end{tabular}
\end{table}
\end{center}

\section*{Appendix: Computation of Bohr-Sommerfeld coefficients}
\renewcommand{\theequation}{A.\arabic{equation}}
\refstepcounter{section}
\renewcommand{\thesection}{A}
\setcounter{equation}{0}


Let $V(x)=\sum_{n \ge 2} a_{n-2} x^n$ with $a_n \in \RR$ and $a_0 > 0.$ We then solve $V(x(y))=y^2$ recursively by looking for an asymptotic expansion $x(y)= \sum_{n \ge 1} \beta_n y^n$ with $\beta_1 > 0$ such that $x_m(y)=\sum_{n =1}^m \beta_n y^n$ satisfies $ \vert V(x_m(y))- y^2 \vert= \mathcal O(\vert y \vert^{m+2}).$ One directly verifies $\beta_1 = a_0^{-1/2}.$ Assume now that we have found $x_m$, then for $x_{m+1}$ we find 
\[ \begin{split} V(x_{m+1}(y)) &= \sum_{n \ge 2} a_n (\beta_{m+1} y^{m+1} + x_m(y))^n \\
&= \sum_{n \ge 2} a_n \sum_{k=0}^n \binom{n}{k} (\beta_{m+1} y^{m+1})^{n-k} x_m(y)^{k} \\ 
&=y^2 +  \sum_{n \ge 2} a_n \sum_{k=0}^{n-1} \binom{n}{k} (\beta_{m+1} y^{m+1})^{n-k} x_m(y)^{k} +  \sum_{n \ge m+2} c_{n,m} y^n
\end{split}\]
for some coefficients $c_{n,m},$ where we used the induction in the last step. For this to satisfy $ \vert V(x_{m+1}(y))- y^2 \vert= \mathcal O(\vert y \vert^{m+3})$ we must eliminate the terms of order $y^{m+2}$ on the right-hand side. These are produced in the middle sum when $n=2$ and $k=1$, and in the right sum when $n=m+2$, so they are eliminated by setting $\beta_{m+1}   := -c_{m+2,m}/\beta_1$.
%
%
This way we find for the positive branch of $V^{-1}(y^2)$ and asymptotic formula 
\begin{equation*}
\begin{aligned}
V^{-1}(y^2)&=a_0^{-1/2}y-\tfrac12a_0^{-2}a_1y^2+(\tfrac58 a_0^{-7/2}a_1^2-\tfrac12 a_0^{-5/2}a_2)y^3\\
&\quad+(-a_0^{-5}a_1^3+\tfrac32a_0^{-4}a_1a_2-\tfrac12a_0^{-3}a_3)y^4+\mathcal O(y^5).
\end{aligned}
\end{equation*}
From this expression, we readily obtain the following result.

\begin{lemm}\label{lem:generalV}
Let $V$ be real analytic such that
\begin{equation*}
V(x)=a_0x^2+a_1x^3+\ldots
\end{equation*}
with $a_0 >0.$ For $x>0$ we set $y^2=V(x)$ and let $V^{-1}$ denote the positive branch of the inverse such that $x=V^{-1}(y^2)$, $V^{-1}(0)=0$, and
$dx/dy=2y/V'(V^{-1}(y^2))$. Then
$$
\frac{2y}{V'(V^{-1}(y^2))}=b_0+b_1y+b_2y^2+b_3y^3+\mathcal O(y^4),
$$
where
$$
\begin{gathered}
b_0=a_0^{-1/2},\quad
b_1=-a_0^{-2}a_1,\quad
b_2=\tfrac{15}{8}a_0^{-7/2}a_1^2-\tfrac32a_0^{-5/2}a_2\\
b_3=-4a_0^{-5}a_1^3+6a_0^{-4}a_1a_2-2a_0^{-3}a_3.
\end{gathered}
$$
\end{lemm}

We use this result to compute $S_0(E)$:

\begin{lemm}\label{lem:S0}
Let $p(x,\xi)=A(\xi)+V(x)$ with $V$ as in Lemma \ref{lem:generalV}, and $A(\xi)=\xi^2$ or $A(\xi)=2(1-\cos\xi)$. Let $S_0(E)=\int_{\{p\le E\}}dx\,d\xi$. If $A(\xi)=\xi^2$ then
\begin{equation}\label{eq:S0xi2}
S_0(E)
=a_0^{-1/2}\pi E+\tfrac14(\tfrac{15}8a_0^{-7/2}a_1^2-\tfrac32a_0^{-5/2}a_2)\pi E^2+\mathcal O(E^3).
\end{equation}
If $A(\xi)=2(1-\cos\xi)$ then
\begin{equation}\label{eq:S0cos}
S_0(E)
=a_0^{-1/2}\pi E+\tfrac14(\tfrac{15}8a_0^{-7/2}a_1^2-\tfrac32a_0^{-5/2}a_2+\tfrac18a_0^{-1/2})\pi E^2+\mathcal O(E^3).
\end{equation}
\end{lemm}

\begin{proof}
We start with the case $A(\xi)=\xi^2$. Then
$$
S_0(E)=\int_{\{p\le E\}}dx\,d\xi=\int_{\{\xi^2+V(x)\le E,\ x>0\}}dx\,d\xi+\int_{\{\xi^2+V(-x)\le E,\ x>0\}}dx\,d\xi,
$$
where the last identity follows by a change of variables. Now apply Lemma \ref{lem:generalV} to each of the two integrals on the right, noting that $\widetilde V(x):=V(-x)=\tilde a_0x^2+\tilde a_1x^3+\ldots$ for $x>0$ with $\tilde a_{2j}=a_{2j}$ and $\tilde a_{2j+1}=-a_{2j+1}$. Changing variables $y^2=V(x)$ and $y^2=\widetilde V(x)$, respectively, we obtain
$$
S_0(E)=\int_{\{\xi^2+y^2\le E,\ y\ge0\}}(b_0+\tilde b_0+(b_1+\tilde b_1)y+(b_2+\tilde b_2)y^2+(b_3+\tilde b_3)y^3+\mathcal O(y^4))\,dy\,d\xi$$
with $b_j$ as in the lemma, and with $\tilde b_j$ defined as $b_j$ but with $a_j$ replaced by $\tilde a_j$. Then $b_0+\tilde b_0=2b_0$, $b_2+\tilde b_2=2b_2$, and $b_1+\tilde b_1=b_3+\tilde b_3=0$. Changing to polar coordinates we get
\begin{equation*}
\begin{aligned}
S_0(E)&
=\int_0^{E^{1/2}}\int_{-\pi/2}^{\pi/2}(2b_0r+2b_2r^3\cos^2t)\,dt\,dr+\mathcal O(E^3) 
=b_0\pi E+\tfrac{1}4b_2\pi E^2+\mathcal O(E^3).
\end{aligned}
\end{equation*}
Inserting the expressions for $b_j$ from Lemma \ref{lem:generalV} gives \eqref{eq:S0xi2}.

Next, consider the case when $A(\xi)=2(1-\cos \xi)$. Make the change of variables $\eta^2=2(1-\cos \xi)$ so that $\xi=\arccos(1-\eta^2/2)$ for $\xi\ge0$. This gives $d\xi=(1-\eta^2/4)^{-1/2}d\eta$, so $d\xi=(1+\eta^2/8+\mathcal O(\eta^4))\,d\eta$ by Taylor's formula. Hence,
$$
\begin{aligned}
S_0(E)&=\int_{\{\eta^2+y^2\le E,\ y\ge0\}}(2b_0+2b_2y^2+\mathcal O(y^4))(1+\eta^2/8+\mathcal O(\eta^4))\,dy\,d\eta
\\&=b_0\pi E+\tfrac{1}4(b_2+\tfrac18b_0)\pi E^2+\mathcal O(E^3).
\end{aligned}
$$
Inserting the expressions for $b_j$ from Lemma \ref{lem:generalV} gives \eqref{eq:S0cos}.
\end{proof}

\begin{lemm}\label{lem:S2}
Let $p(x,\xi)=A(\xi)+V(x)$ with $V$ as in Lemma \ref{lem:generalV}, and $A(\xi)=\xi^2$ or $A(\xi)=2(1-\cos\xi)$. Let $S_2(E)=-\frac1{24}\partial_E^2\int_{\{p\le E\}}A''(\xi)V''(x)\,dx\,d\xi$. If $A(\xi)=\xi^2$ then
\begin{equation}\label{eq:S2xi2}
S_2(E)=\tfrac{1}{24}(\tfrac{21}4a_0^{-5/2}a_1^2-9a_0^{-3/2}a_2)\pi+\mathcal O(E).
\end{equation}
If $A(\xi)=2(1-\cos\xi)$ then
\begin{equation}\label{eq:S2cos}
S_2(E)=\tfrac{1}{24}(\tfrac{21}4a_0^{-5/2}a_1^2-9a_0^{-3/2}a_2+\tfrac34a_0^{1/2})\pi+\mathcal O(E).
\end{equation}
\end{lemm}

\begin{proof}
We start with the case $A(\xi)=\xi^2$. We compute $S_2(E)=-\frac1{24}\partial_E^2 I(E)$, where after a change of variable
$$
I(E)=\int_{\{\xi^2+V(x)\le E,\ x>0\}}2V''(x)\,dx\,d\xi+\int_{\{\xi^2+V(-x)\le E,\ x>0\}}2V''(-x)\,dx\,d\xi.
$$
As in the proof of Lemma \ref{lem:S0}, we change variables $y^2=V(x)$ and $y^2=V(-x)$, respectively. Using Lemma \ref{lem:generalV} and writing $V''(V^{-1}(y^2))=\sum_{j=0}^3c_j y^j+\mathcal O(y^4)$, we get
$$
\begin{gathered}
c_0=2a_0,\quad c_1=6a_0^{-1/2}a_1,\quad c_2=-3a_0^{-2}a_1^2+12a_0^{-1}a_2,\\ c_3=\tfrac{30}8a_0^{-7/2}a_1^3-15a_0^{-5/2}a_1a_2+20a_0^{-3/2}a_3,
\end{gathered}
$$
and $I(E)=I_+(E)+I_-(E)$, where
\begin{align*}
I_\pm(E)&=\int_{\{\xi^2+y^2\le E,\ y\ge0\}}2\bigg(\sum_{j=0}^3(\pm1)^jc_jy^j+\mathcal O(y^4)\bigg)\bigg(\sum_{j=0}^3(\pm1)^jb_jy^j+\mathcal O(y^4)\bigg)\,dy\,d\xi,
\end{align*}
with $b_j$ as in the lemma.
Performing the multiplication and computing $I_++I_-$ we get
$$
I(E)=\int_{\{\xi^2+y^2\le E,\ y\ge0\}}(4b_0c_0+4(b_0c_2+b_1c_1+b_2c_0)y^2+\mathcal O(y^4))\,dy\,d\xi.
$$
Changing to polar coordinates gives
\begin{equation*}
\begin{aligned}
I(E)&=\int_0^{E^{1/2}}\int_{-\pi/2}^{\pi/2}(4b_0c_0r+4(b_0c_2+b_1c_1+b_2c_0)r^3\cos^2t)\,dt\,dr+\mathcal O(E^3)\\&
=2b_0c_0\pi E+\tfrac{1}2(b_0c_2+b_1c_1+b_2c_0)\pi E^2+\mathcal O(E^3).
\end{aligned}
\end{equation*}
Hence,
$$
S_2(E)=-\tfrac1{24}I''(E)=-\tfrac{1}{24}(b_0c_2+b_1c_1+b_2c_0)\pi+\mathcal O(E).
$$
Inserting the expressions for $b_j$ and $c_j$ we obtain \eqref{eq:S2xi2}.

Next, consider the case when $A(\xi)=2(1-\cos \xi)$. Change variables from $x$ to $y$ as above, and then make the change of variables $\eta^2=2(1-\cos \xi)$ so that $\xi=\arccos(1-\eta^2/2)$ for $\xi\ge0$, $d\xi=(1+\eta^2/8+\mathcal O(\eta^4))\,d\eta$, and $A''(\xi)=2\cos\xi=2-\eta^2$. This gives
$$
\begin{aligned}
I(E)&=\int_{\{\eta^2+y^2\le E,\ y\ge0\}}\big[2b_0c_0+2(b_0c_2+b_1c_1+b_2c_0)y^2+\mathcal O(y^4)\big]\\&\qquad\qquad\qquad\qquad\times (2-\eta^2)(1+\eta^2/8+\mathcal O(\eta^4))\,dy\,d\eta.
\end{aligned}
$$
Here, $(2-\eta^2)(1+\eta^2/8+\mathcal O(\eta^4))=2(1-\frac38\eta^2+\mathcal O(\eta^4))$ so changing to polar coordinates and computing the resulting integrals we obtain
\begin{equation*}
\begin{aligned}
I(E)=2b_0c_0\pi E+\tfrac{1}2((b_0c_2+b_1c_1+b_2c_0)-\tfrac3{8}b_0c_0)\pi E^2+\mathcal O(E^3),
\end{aligned}
\end{equation*}
and
$$
S_2(E)=-\tfrac1{24}I''(E)=-\tfrac{1}{24}(b_0c_2+b_1c_1+b_2c_0-\tfrac38b_0c_0)\pi+\mathcal O(E).
$$
Inserting the expressions for $b_j$ and $c_j$ gives \eqref{eq:S2cos}.
\end{proof}


\smallsection{Acknowledgements} 
We would like to thank Isabel Detherage, Nikhil Srivastava and Zach Stier
for interesting discussions and for bringing this topic to our attention. 
SB acknowledges support by the SNF Grant PZ00P2 216019, JW was partially supported by the Swedish Research Council grants 2019-04878 and 2023-04872,
and MZ by the Simons Targeted Grant Award No.~896630.

\end{document}